\documentclass[11pt]{amsart}

\usepackage[mathscr]{euscript}
\usepackage{bbm}
\usepackage{amssymb}
\usepackage{graphicx}
\usepackage{pb-diagram,pb-xy} 
\usepackage[all,cmtip,arrow]{xy} 
\usepackage[margin=1in]{geometry}  
\usepackage{etoolbox} 

\numberwithin{equation}{section}

\theoremstyle{plain}
\newtheorem{theorem}[equation]{Theorem}

\newtheorem{lemma}[equation]{Lemma}

\newtheorem{corollary}[equation]{Corollary}
\newtheorem{proposition}[equation]{Proposition}
\newtheorem{conjecture}[equation]{Conjecture}

\newtheorem{claim}[equation]{Claim}
\newtheorem*{claim*}{Claim}

\theoremstyle{definition}
\newtheorem{definition}[equation]{Definition}

\newtheorem{remark}[equation]{Remark}

\newtheorem{example}[equation]{Example}
\newtheorem{examples}[equation]{Examples}

\newcommand{\isom}{\cong}                       
\newcommand{\homeq}{\simeq}                     
\newcommand{\smsh}{\wedge}                      
\newcommand{\wdge}{\vee}                        

\newcommand{\homsmsh}{\mathbin{\tilde{\smsh}}}

\newcommand{\Smsh}{\bigwedge}                   
\newcommand{\Wdge}{\bigvee}                     

\newcommand{\Q}{\mathbf{Q}}
\newcommand{\N}{\mathsf{N}}                     
\newcommand{\M}{\mathsf{M}}
\newcommand{\A}{\mathsf{A}}
\newcommand{\B}{\mathsf{B}}
\newcommand{\one}{\mathsf{1}}

\renewcommand{\P}{\mathsf{P}}
\newcommand{\Sig}{\mathsf{\Sigma}}

\newcommand{\E}{\mathsf{E}}
\newcommand{\K}{\mathsf{K}}
\renewcommand{\d}{\mathsf{d}}
\newcommand{\Com}{\mathsf{Com}}

\newcommand{\C}{\mathbf{C}}
\newcommand{\U}{\mathbf{U}}
\newcommand{\R}{\mathbf{R}}
\newcommand{\X}{\mathbf{X}}
\renewcommand{\L}{\mathbf{L}}
\renewcommand{\c}{\mathbf{c}}

\newcommand{\colim}{\operatorname{colim} }
\newcommand{\hocolim}{\operatorname{hocolim} }

\newcommand{\Map}{\operatorname{Map} }
\newcommand{\Nat}{\operatorname{Nat} }

\newcommand{\Hom}{\operatorname{Hom} }
\newcommand{\Ext}{\operatorname{Ext} }

\newcommand{\Mod}{\operatorname{Mod} }
\newcommand{\Comod}{\operatorname{Comod} }
\newcommand{\Coalg}{\operatorname{Coalg} }
\newcommand{\Emb}{\operatorname{Emb} }

\DeclareMathOperator*{\hofib}{hofib}

\DeclareMathOperator*{\thofib}{thofib}
\DeclareMathOperator*{\thocofib}{thocofib}
\DeclareMathOperator*{\tcofib}{tcofib}

\newcommand{\un}[1]{\underline{#1}}

\newcommand{\spaces}{\mathscr{T}op}
\newcommand{\based}{\spaces_*}
\newcommand{\spectra}{{\mathscr{S}p}}              
\newcommand{\finbased}{{\mathscr{T}op^{\mathsf{f}}_*}}

\newcommand{\weq}{\; \tilde{\longrightarrow} \;}      
\newcommand{\epi}{\twoheadrightarrow}           

\newcommand{\der}{\partial}                     

\newcommand{\ord}[1]{#1\textsuperscript{th}}

\newcommand{\fMfld}{\mathsf{fMfld}}

\begin{document}

\title{Manifolds, K-theory and the calculus of functors}
\author{Gregory Arone and Michael Ching}

\date{\today}

\begin{abstract}
The Taylor tower of a functor from based spaces to spectra can be classified according to the action of a certain comonad on the collection of derivatives of the functor. We describe various equivalent conditions under which this action can be lifted to the structure of a module over the Koszul dual of the little $L$-discs operad. In particular, we show that this is the case when the functor is a left Kan extension from a certain category of `pointed framed $L$-manifolds' and pointed framed embeddings. As an application we prove that the Taylor tower of Waldhausen's algebraic K-theory of spaces functor is classified by an action of the Koszul dual of the little $3$-discs operad.
\end{abstract}

\maketitle

In previous work \cite{arone/ching:2014,arone/ching:2014b} we described the structure possessed by the Goodwillie derivatives of a functor from the category of based spaces, $\based$, to the category of spectra, $\spectra$. This structure is that of a symmetric sequence with coaction of a certain comonad $\C$. The comonad $\C$ can be described as the homotopy colimit of a sequence
\[ \C_0 \to \C_1 \to \C_2 \to \dots \]
of comonads where a coaction by $\C_L$ gives a symmetric sequence precisely the structure of a right module over the operad $\K\E_L$ (the Koszul dual of the stable reduced little $L$-discs operad).

It follows that any right $\K\E_L$-module $\M$ gives rise to a $\C$-coalgebra, via the canonical comonad map $\C_L \to \C$, and hence to some functor $F$ whose derivatives are precisely the original $\M$. If this is the case, let us say that \emph{the derivatives of $F$ are a $\K\E_L$-module}. This construction is natural, so that a morphism of $\K\E_L$-modules induces a natural transformation between the corresponding functors.

The above construction suggests several tasks: most notable being to decide which functors $F$, and which natural transformations, arise in this way for a given $L$. One answer to this is given by the following theorem, which is the main result of Section~\ref{sec:KE_L}.

\begin{theorem} \label{thm:intro1}
Let $F: \based \to \spectra$ be a finitary polynomial functor and $L \geq 0$ an integer. Then the following are equivalent:
\begin{enumerate}
  \item the derivatives of $F$ are a $\K\E_L$-module;
  \item there is an $\E_L$-comodule $\N$ such that, for $X \in \based$:
  \[ F(X) \homeq \N \smsh_{\E_L} \Sigma^\infty X^{\smsh *}; \]
  where the symmetric sequence $\Sigma^\infty X^{\smsh *}$ becomes an $\E_L$-module via the diagonal maps on $X$;
  \item $F$ is the left Kan extension, along the inclusion $\Gamma_L \to \based$, of a functor $G: \Gamma_L \to \spectra$, where $\Gamma_L$ is the subcategory of $\based$ described in Definition~\ref{def:Gamma-L} below.
\end{enumerate}
\end{theorem}

In particular, one can view the formula in part (2) as an explicit description of how a polynomial functor can be pieced together from homogeneous pieces. The special case where $L = 0$ tells us that the $\C$-coalgebra structure on the derivatives of $F$ is trivial if and only if $F$ splits as a product of the layers of its Taylor tower. Part (3) of the theorem then tells us that this is the case if and only if $F$ is the left Kan extension of a functor defined on $\Gamma_0$, which, as we shall see, is the category of finite pointed sets and pointed functions that are injective away from the basepoint.

In Section~\ref{sec:fMfld} we use Theorem~\ref{thm:intro1} to give conditions under which the derivatives of a representable functor $\Sigma^\infty \Hom_{\based}(X,-)$ are a $\K\E_L$-module. We show that this is the case when $X$ is a `pointed framed manifold' of dimension $L$. This basically means that $X$ is a finite based cell complex such that removing the basepoint leaves behind a framed manifold. For example, the sphere $S^L$ has this property and so the derivatives of the functor $\Sigma^\infty \Omega^L(-)$ are a $\K\E_L$-module.

Our construction of a $\K\E_L$-module structure on the derivatives of $\Sigma^\infty \Hom_{\based}(X,-)$ is functorial in `pointed framed embeddings' in the variable $X$. This means that if $f: X \to Y$ is a pointed map whose restriction to $f^{-1}(Y - \{y_0\})$ preserves the framing (up to scalar multiples), then the induced map
\[ f^* : \der_*(\Sigma^\infty \Hom_{\based}(Y,-)) \to \der_*(\Sigma^\infty \Hom_{\based}(X,-)) \]
can be realized as a map of $\K\E_L$-modules. We then have the following consequence which extends part (3) of Theorem~\ref{thm:intro1}.

\begin{theorem} \label{thm:intro2}
Let $\fMfld^L_*$ denote the category of pointed framed manifolds and pointed framed embeddings. Let $F: \based \to \spectra$ be a finitary homotopy functor. If $F$ is the homotopy left Kan extension of a pointed functor $G: \fMfld^L_* \to \spectra$ along the forgetful functor $\fMfld^L_* \to \based$, then the derivatives of $F$ are a $\K\E_L$-module. If $F$ is polynomial, the converse of the above statement also holds.
\end{theorem}

Section~\ref{sec:A} contains an application of the previous results, where we show that the derivatives of Waldhausen's algebraic K-theory of spaces functor $A(X)$ are a $\K\E_3$-module.

\subsection*{Acknowledgements}
The second author would like to thank Andrew Blumberg for useful conversations concerning Waldhausen's algebraic K-theory of spaces, and David Ayala for sharing some of his work with John Francis on zero-pointed manifolds. Work for this article was carried out while the second author was a Research Member at the Mathematical Sciences Research Institute in Spring 2014. The second author was supported by the National Science Foundation via grant DMS-1308933.

\section{Background}

Let $\finbased$ be the category of finite based CW-complexes and $\spectra$ the category of $S$-modules of EKMM \cite{elmendorf/kriz/mandell/may:1997}. We let $[\finbased,\spectra]$ denote the category of pointed, simplicially-enriched functors $F: \finbased \to \spectra$. Recall that a functor $F$ is \emph{$n$-excisive} if it takes strongly homotopy cocartesian $(n+1)$-cubes to homotopy cartesian cubes, and we say that $F$ is \emph{polynomial} if it is $n$-excisive for some integer $n$.

Each functor $F \in [\finbased,\spectra]$ has a \emph{Taylor tower}
\[ F \to \dots \to P_nF \to P_{n-1}F \to \dots P_1F \to P_0F = * \]
where $P_nF \in [\finbased,\spectra]$ is $n$-excisive and the natural transformation $F \to P_nF$ is initial, up to homotopy, among maps from $F$ to an $n$-excisive functor.

The \emph{layers} of the Taylor tower are the functors
\[ D_nF := \hofib(P_nF \to P_{n-1}F) \]
and these can be expressed as
\[ D_nF(X) \homeq (\der_nF \smsh (\Sigma^\infty X)^{\smsh n})_{h\Sigma_n} \]
where $\der_nF$ is a spectrum with (naive) $\Sigma_n$-action called the \emph{\ord{$n$} derivative} of $F$.

In \cite{arone/ching:2014b} we described three related classifications of polynomial functors from based spaces to spectra: in terms of
\begin{enumerate}
  \item coalgebras over a comonad $\C$ on symmetric sequences;
  \item comodules over the commutative operad $\Com$;
  \item Kan extensions from the category of pointed finite sets and pointed functions.
\end{enumerate}
We recall the details of each of these approaches.

Let $\mathbb{E}_L$ be the standard little $L$-discs operad in based spaces, where $\mathbb{E}_L(n)$ is the space of `standard' embeddings $f: \coprod_n D^L \to D^L$, with a disjoint basepoint, where operad composition is given by composition of embeddings. Let $\E_L$ denote the stable reduced version of $\mathbb{E}_L$ given by
\[ \E_L(n) := \begin{cases}
  \Sigma^\infty \mathbb{E}_L(n) & \text{if $n > 0$}; \\
  \quad * & \text{if $n = 0$}.
\end{cases} \]
Let $\K\E_L$ be (a cofibrant model for) the \emph{Koszul dual operad} of $\E_L$ as in \cite{ching:2012}.

Associated to the operad $\K\E_L$ is a comonad $\C_L$, on the category of symmetric sequences (in $\spectra$), given by
\[ \C_L(\A)(k) := \prod_{n} \left[ \prod_{\un{n} \epi \un{k}} \Map(\K\E_L(n_1) \smsh \dots \smsh \K\E_L(n_k), \A(n)) \right]^{\Sigma_n}. \]
The comonad structure is determined by the operad structure on $\K\E_L$ and a $\C_L$-coalgebra is precisely a right $\K\E_L$-module.

The standard inclusions $D^0 \subseteq D^1 \subseteq D^2 \subseteq \dots$ determine a sequence of operads
\[ \E_0 \to \E_1 \to \E_2 \to \cdots \]
which, in turn, induces a dual sequence of operads
\[ \K\E_0 \leftarrow \K\E_1 \leftarrow \K\E_2 \leftarrow \cdots \]
and a sequence of comonads
\[ \C_0 \to \C_1 \to \C_2 \to \cdots . \]
Let $\C$ be the objectwise homotopy colimit of this sequence. Then $\C$ inherits a comonad structure. In \cite{arone/ching:2014b} we showed that the derivatives $\der_*F$ of a functor $F: \finbased \to \spectra$ have a coaction by $\C$ and that $\der_*$ sets up an equivalence between the homotopy theory of polynomial functors $F$ and that of bounded $\C$-coalgebras.

The second classification of polynomial functors $\based \to \spectra$ (due to Dwyer and Rezk) is via functors $\mathsf{\Omega} \to \spectra$ where $\mathsf{\Omega}$ is the category of nonempty finite sets and surjections. We think of such functors as comodules over the commutative operad $\Com$ in the category of spectra, given by $\Com(n) = S$ for all $n \geq 1$.

For any operad $\P$ of spectra, a \emph{$\P$-comodule} is a symmetric sequence $\N$ together with a structure map
\[ \N(n) \smsh \P(n_1) \smsh \dots \smsh \P(n_k) \to \N(k) \]
for each surjection $\un{n} \epi \un{k}$. Here $\un{n}$ denotes the finite set $\{1,\dots,n\}$ and, for a surjection $\alpha: \un{n} \epi \un{k}$, we write $n_i := |\alpha^{-1}(i)|$. A $\Com$-comodule is the same as a functor $\mathsf{\Omega} \to \spectra$.

Associated to any functor $F: \finbased \to \spectra$ is a $\Com$-comodule $\N[F]$ given by
\[ \N[F](k) := \Nat_{X \in \finbased}(\Sigma^\infty X^{\smsh *}, F(X)) \]
with $\Com$-comodule structure arising from the $\Com$-module structure on $\Sigma^\infty X^{\smsh *}$ that is induced by the diagonal on the based space $X$.

Dwyer and Rezk showed that a polynomial $F$ can be recovered from the $\Com$-comodule $\N[F]$ via an equivalence
\[ \N[F] \homsmsh_{\Com} \Sigma^\infty X^{\smsh *} \weq F(X). \]
Here $\homsmsh_{\Com}$ denotes the homotopy coend of $\Com$-comodule and a $\Com$-module. This construction sets up an equivalence between the homotopy theory of polynomial functors $\finbased \to \spectra$ and that of bounded $\Com$-comodules. (Dwyer and Rezk's work is not published; an account of this can be found in \cite{arone/ching:2014b}).

For the third classification of polynomial functors from based spaces to spectra, let $\Gamma$ denote the category of finite pointed sets and basepoint-preserving functions. Any polynomial functor $F: \finbased \to \spectra$ is equivalent to the homotopy left Kan extension of its restriction to the subcategory $\Gamma \subseteq \finbased$. More precisely, this sets up an equivalence, for each $n$, between the homotopy theory of $n$-excisive functors $\finbased \to \spectra$, and that of pointed functors $\Gamma^{\leq n} \to \spectra$ where $\Gamma^{\leq n}$ is the full subcategory of $\Gamma$ whose objects are the finite pointed sets of cardinality (excluding the basepoint) at most $n$.

We can summarize these classifications and their relationships in the following diagram which commutes up to natural equivalence, and in which each functor induces an equivalence of homotopy theories:
\begin{equation} \label{eq:classification} \begin{diagram} \dgARROWLENGTH=4em
  \node[2]{\Coalg_{\mathsf{b}}(\C)} \\
  \node{\Comod_{\mathsf{b}}(\Com)} \arrow{ne,t}{\hocolim_L - \homsmsh_{\E_L} \B_{L}} \arrow[2]{e,t}{- \homsmsh_{\Com} \Sigma^\infty X^{\smsh *}} \arrow{se,b}{\L}
    \node[2]{[\finbased,\spectra]_{\mathsf{poly}}} \arrow{nw,t}{\der_*} \\
  \node[2]{[\Gamma,\spectra]_{\mathsf{b}}}  \arrow{ne,b}{\mathsf{hLKan}}
\end{diagram} \end{equation}
Here $[\Gamma,\spectra]$ denotes the category of pointed functors $\Gamma \to \spectra$, and $[\Gamma,\spectra]_{\mathsf{b}}$ is the subcategory of those functors (and natural transformations) that are left Kan extensions along $\Gamma^{\leq n} \subseteq \Gamma$ for some $n$.  and the subscripts `$\leq n$' denote the subcategories of $n$-truncated objects. The subscripts $\mathsf{b}$ on the left and top objects denote the categories of bounded $\Com$-comodules and bounded $\C$-coalgebras respectively. The arrow marked $\mathsf{hLKan}$ is homotopy left Kan extension along the inclusion $\Gamma \subseteq \finbased$.

We have introduced the three arrows $\der_*$, $- \homsmsh_{\Com} \Sigma^\infty X^{\smsh *}$ and $\mathsf{hLKan}$ in (\ref{eq:classification}), so now let us describe the other two functors in this picture.

First note that a $\Com$-comodule $\N$ inherits the structure of a $\E_L$-comodule for any $L$ via pullback along the operad map $\E_L \to \Com$. In \cite{arone/ching:2014b} we introduced a certain bisymmetric sequence $\B_{L}$ with an $\E_L$-module action on one variable and a $\K\E_L$-module action on the other. Forming the homotopy coend of an $\E_L$-comodule $\N$ with this $\E_L$-module results in a $\K\E_L$-module which we think of as a certain `Koszul dual' of $\N$. When $\N$ is a $\Com$-comodule, there is an induced sequence of homotopy coends
\[ \N \homsmsh_{\E_0} \B_{\E_0} \to \N \homsmsh_{\E_1} \B_{\E_1} \to \dots \]
whose homotopy colimit is a $\C$-coalgebra. This construction gives the top-left hand functor in (\ref{eq:classification}).

We think of the bottom-left functor in (\ref{eq:classification}) as part of a `Pirashvili'-type equivalence between the category of functors $\mathsf{\Omega} \to \spectra$ and the category of pointed functors $\Gamma \to \spectra$ where $\Gamma$ is the category of all finite pointed sets and pointed functions. Explicitly, this functor $\L$ is given by
\[ \L(\N)(J_+) = \N \homsmsh_{\Com} \Sigma^\infty (J_+)^{\smsh *}, \]
which is just the restriction of the middle horizontal functor in (\ref{eq:classification}) to the subcategory $\Gamma^{\leq n}$ of $\finbased$.

\section{Functors whose derivatives are a module over the Koszul dual of the little discs operad} \label{sec:KE_L}

We now turn to a description of those functors $F: \finbased \to \spectra$ for which the $\C$-coalgebra structure on $\der_*F$ arises from a $\K\E_L$-module structure for some given $L$. We have three characterizations that correspond to the three classifications of polynomial functors in (\ref{eq:classification}). Our main result is as follows.

\begin{theorem} \label{thm:KE-mod-functor}
Let $F: \finbased \to \spectra$ be a polynomial functor and $L \geq 0$ an integer. Then the following are equivalent:
\begin{enumerate}
  \item the derivatives $\der_*F$ have a $\K\E_L$-module structure, from which $F$ can then be reconstructed by
  \[ F(X) \homeq \widetilde{\Map}_{\C}(\der_*R_X,\U(\der_*F)) \]
  where $\U$ denotes the forgetful functor from $\K\E_L$-modules to $\C$-coalgebras, and $\widetilde{\Map}_{\C}(-,-)$ is the derived mapping spectrum for two $\C$-coalgebras;
  \item there is a bounded $\E_L$-comodule $\N$ and a natural equivalence
  \[ F(X) \homeq \N \homsmsh_{\E_L} \Sigma^\infty X^{\smsh *}; \]
  \item there is a simplicially-enriched functor $G: \Gamma_L \to \spectra$ and an equivalence
  \[ \mathsf{hLKan}(G) \homeq F \]
  where $\Gamma_L$ is the subcategory of $\finbased$ introduced in Definition~\ref{def:Gamma-L} below, and $\mathsf{hLKan}$ denotes the enriched homotopy left Kan extension along the inclusion $\Gamma_L \to \finbased$.
\end{enumerate}
\end{theorem}

The category $\Gamma_L$ is defined as follows.

\begin{definition} \label{def:Gamma-L}
For an integer $L \geq 0$, let $D^L_+$ denote the standard closed unit disc in $\mathbb{R}^L$ with a disjoint basepoint. Then we denote by $\Gamma_L$ the (non-full) subcategory of $\finbased$ whose objects are the based spaces of the form
\[ \Wdge_{I} D^L_+ \]
for a finite set $I$, and whose morphisms are the continuous basepoint-preserving functions
\[ f: \Wdge_{I} D^L_+ \to \Wdge_{J} D^L_+ \]
with the following properties:
\begin{itemize}
  \item for each $i \in I$, the restriction $f_i$ of $f$ to the corresponding copy of $D^L$ is: \emph{either} a standard affine embedding (i.e. of the form $x \mapsto ax+b$ for some $a > 0$ and $b \in D^L$) into one of the copies of $D^L$ in the target; \emph{or} the constant map to the basepoint;
  \item the images of the non-constant maps $f_i$ are disjoint.
\end{itemize}
For convenience, we often write $I_+$ for the object of $\Gamma_L$ given by the based space $\Wdge_{I} D^L_+$. The category $\Gamma_L$ is topologically-enriched, with mapping spaces inherited from those of $\based$. These mapping spaces can be written
\begin{equation} \label{eq:Gamma_L} \Gamma_L(I_+,J_+) = \Wdge_{\alpha: I_+ \to J_+} \Smsh_{j \in J} \mathbb{E}_L(I_j) \end{equation}
where the coproduct is taken over the set of pointed maps $\alpha: I_+ \to J_+$, where $I_j := \alpha^{-1}(j)$, and where $\mathbb{E}_L$ denotes the \emph{non-reduced} little $L$-discs operad in based spaces (with $\mathbb{E}_L(0) = S^0$). This description says also that $\Gamma_L$ is the PROP associated to the operad $\mathbb{E}_L$.

We write $[\Gamma_L,\spectra]$ for the category of pointed topologically-enriched functors from $\Gamma_L$ to $\spectra$.
\end{definition}

We prove Theorem~\ref{thm:KE-mod-functor} by constructing a diagram of categories and functors of the following form:
\begin{equation} \label{eq:KE-mod-functor} \begin{diagram} \dgARROWLENGTH=4em
  \node[2]{\Mod(\K\E_L)} \arrow{se,t}{\Map_{\C}(\der_*(R_X),\U(-))} \\
  \node{\Comod(\E_L)} \arrow{ne,t}{\Q} \arrow[2]{e,t}{- \homsmsh_{\E_L} \Sigma^\infty X^{\smsh *}} \arrow{se,b}{\L}
    \node[2]{[\finbased,\spectra]} \\
  \node[2]{[\Gamma_L,\spectra]}  \arrow{ne,b}{\mathsf{hLKan}}
\end{diagram} \end{equation}
This diagram should be compared to (\ref{eq:classification}) to which it is closely related.

The three functors whose target is $[\finbased,\spectra]$ correspond to the three ways to construct a functor $F: \finbased \to \spectra$ described in the statement of Theorem~\ref{thm:KE-mod-functor}. The required result follows by establishing that the other two functors in this diagram induce equivalences of homotopy categories on bounded objects, and that the whole diagram commutes up to natural equivalence.

\subsection*{Koszul duality between bounded $\E_L$-comodules and bounded $\K\E_L$-modules}

First note that we can conveniently describe modules and comodules over an operad $\P$ in terms of functors defined on the PROP associated to $\P$. If $\P$ is an operad of spectra, we have a $\spectra$-enriched category $\un{\P}$ with objects the nonempty finite sets, and with mapping spectra given by
\begin{equation} \label{eq:PROP} \un{\P}(I,J) := \Wdge_{I \epi J} \Smsh_{j \in J} \P(I_j). \end{equation}
Here we take the coproduct over all surjections $\alpha: I \epi J$, and, for $j \in J$, write $I_j := \alpha^{-1}(j)$ with $\alpha$ understood. The composition and unit maps for the category $\un{\P}$ come directly from the composition and unit for the operad $\P$.

A (right) $\P$-module can now be identified with a $\spectra$-enriched functor $\un{\P}^{op} \to \spectra$, and a $\P$-comodule with a $\spectra$-enriched functor $\un{\P} \to \spectra$. We freely use these identifications throughout this paper.

One further note: when the operad $\P$ arises by taking the suspension spectrum of an operad $\mathbb{P}$ of unbased spaces, i.e. by $\P(n) = \Sigma^\infty \mathbb{P}(n)_+$, we can, similarly, form a $\spaces$-enriched category $\un{\mathbb{P}}$ (defined in the same way as $\un{\P}$. A $\P$-module can then be identified with a $\spaces$-enriched functor $\un{\mathbb{P}} \to \spectra$.

We now describe the top-left functor in (\ref{eq:KE-mod-functor}) and show that it induces an equivalence of homotopy theories. More precisely, we construct a Quillen adjunction (where both sides have the projective model structure)
\[ \Q : \Comod(\E_L) \rightleftarrows \Mod(\K\E_L) : \R \]
and show that it restricts to an equivalence on the subcategories of bounded objects. This adjunction is a form of Koszul duality and is constructed via the methods of \cite{arone/ching:2014b} which we now recall.

The key construction, made in \cite[Def.~3.38 and after]{arone/ching:2014b} is of a certain bisymmetric sequence $\B_{L}$, associated to the operad $\E_L$, that has an $\E_L$-module structure on its first variable, and a $\K\E_L$-module structure on its second variable. Each of these structures, when taken individually, is equivalent to a trivial structure, but when considered together they are not. Explicitly we have
\begin{equation} \label{eq:BL} \B_L(I,J) := \prod_{I \epi J} \Smsh_{j \in J} \B(\one,\E_L,\E_L)(I_j) \end{equation}
where the product is taken over all surjections between the finite sets $I$ and $J$, and for such a surjection $\alpha$ and $j \in J$ we use $I_j$ to denote $\alpha^{-1}(j)$.

We can think of $\B_L$ as an $\E_L$-module in the category of $\K\E_L$-modules, in other words as a $\spectra$-enriched functor $\un{\E}_L^{op} \to \Mod(\K\E_L)$, or equivalently, a $\spaces$-enriched functor
\[ \B_L: \un{\mathbb{E}}_L^{op} \to \Mod(\K\E_L). \]
The category $\Mod(\K\E_L)$ has a projective model structure that is compatible with its enrichment in $\spaces$, so has a topologically-enriched cofibrant replacement functor $\c: \Mod(\K\E_L) \to \Mod(\K\E_L)$. Write $\tilde{\B}_L$ for the composite
\[ \un{\mathbb{E}}_L^{op} \arrow{e,t}{\B_L} \Mod(\K\E_L) \arrow{e,t}{\c} \Mod(\K\E_L). \]
In other words, $\tilde{\B}_L$ is a replacement for $\B_L$ (as an $\E_L$-module in the category of $\K\E_L$-modules) and has the property that, for each finite set $I$, the $\K\E_L$-module $\tilde{\B}_L(I,-)$ is cofibrant in the projective model structure.

We now define our left adjoint
\[ \Q: \Comod(\E_L) \to \Mod(\K\E_L) \]
by
\[ \Q(\N) := \N \smsh_{\E_L} \tilde{\B}_{L}. \]
This is an enriched coend formed over the category $\un{\E}_L$ and it inherits a $\K\E_L$-module structure from that on $\tilde{\B}_L$. The functor $\Q$ has right adjoint
\[ \R: \Mod(\K\E_L) \to \Comod(\E_L) \]
given by
\[ \R(\M) := \Map_{\K\E_L}(\tilde{\B}_{L},\M) \]
with $\E_L$-comodule structure arising from the $\E_L$-module structure on $\tilde{\B}_{L}$. The right adjoint preserves fibrations and trivial fibrations since each $\tilde{\B}_{L}(I,-)$ is a cofibrant $\K\E_L$-module, so we have a Quillen adjunction.

Note that the left adjoint here is a model for the `derived indecomposables' of an $\E_L$-comodule, and the right adjoint can be viewed as a model for the `derived primitives' of a $\K\E_L$-module.

\begin{proposition} \label{prop:koszul-EL}
The Quillen adjunction
\[ \Q : \Comod(\E_L) \rightleftarrows \Mod(\K\E_L) : \R \]
restricts to an equivalence between the homotopy categories of bounded $\E_L$-comodules and bounded $\K\E_L$-modules.
\end{proposition}
\begin{proof}
We first show that $\Q$ and $\R$ preserve boundedness (up to equivalence). Suppose $\N$ is an $n$-truncated $\E_L$-comodule (i.e. $\N(k) = *$ for $k > n$). Then $\Q(\N)$ is equivalent, as a symmetric sequence, to a derived coend
\[ \N \homsmsh_{\E_L} \un{\one} \]
whose \ord{$k$} term can be written as the geometric realization of a simplicial object with $r$-simplices
\[ \Wdge_{n_0 \geq \dots n_r \geq k} \N(n_0) \smsh_{\Sigma_{n_0}} \un{\E}_L(n_0,n_1) \smsh_{\Sigma_{n_1}} \dots \smsh_{\Sigma_{n_{r-1}}} \un{\E_L}(n_{r-1},n_r) \smsh_{\Sigma_{n_r}} \un{\one}(n_r,k). \]
Here $\un{\one}$ is the `trivial' bisymmetric sequence with $\un{\one}(I,J) = \Sigma^\infty \mathsf{Bij}(I,J)_+$, where $\mathsf{Bij}(I,J)$ is the set of bijections from $I$ to $J$. This \ord{$k$} term is trivial if $k > n$ and so $\Q(\N)$ is also $n$-truncated.

Similarly, if $\M$ is an $n$-truncated $\K\E_L$-module, then $\R(\M)$ is equivalent to the derived end
\[ \widetilde{\Map}_{\K\E_L}(\un{\one},\M) \]
which is, by a similar calculation, also $n$-truncated. It follows that $\Q$ and $\R$ determine an adjunction between the homotopy categories of bounded $\E_L$-comodules and bounded $\K\E_L$-modules.

We now prove that this adjunction is an equivalence by showing that the derived unit and counit of the adjunction are equivalences when applied to bounded objects. First take a bounded $\E_L$-comodule $\N$. Our goal is to show that the derived unit $\eta: \N \to \R\Q\N$ is an equivalence. To see this let $\N_{\leq n}$ denote the \emph{$n$-truncation} of $\N$, defined by
\[ \N_{\leq n}(k) := \begin{cases}
  \N(k) & \text{if $k \leq n$}; \\
  * & \text{if $k > n$}.
\end{cases} \]
There are natural fibre sequences of $\E_L$-comodules
\[ \N_{\leq (n-1)} \to \N_{\leq n} \to \N_{=n}, \]
where $\N_{=n}$ has only its \ord{$n$} non-trivial, and equal to $\N(n)$. Since both $\Q$ and $\R$ preserve fibre sequences (which are the same as cofibre sequences in the stable categories $\Comod(\E_L)$ and $\Mod(\K\E_L)$), we can use these sequences to reduce to the case that $\N$ is concentrated in a single term, and, in particular, has a trivial $\E_L$-comodule structure.

Proposition~3.67 of \cite{arone/ching:2014b} tells us that when $\N$ is a bounded trivial $\E_L$-comodule, there is an equivalence of $\K\E_L$-modules
\begin{equation} \label{eq:indec-triv} \N \smsh_{\E_L} \tilde{\B}_{L} \homeq \C_{L}(\N) \end{equation}
where the right-hand side can be thought of as the `cofree' $\K\E_L$-module generated by $\N$.

Now consider the following diagram (of symmetric sequences):
\begin{equation} \label{eq:koszul-unit} \begin{diagram}
  \node{\N} \arrow{e,t}{\eta} \arrow{s,r}{\sim} \node{\Map_{\K\E_L}(\tilde{\B}_{L}, \N \smsh_{\E_L} \tilde{\B}_{L})} \arrow{s,r}{\sim} \\
  \node{\Map_{\Sig}(\tilde{\B}_{L},\N)} \arrow{e,b}{\isom} \node{\Map_{\K\E_L}(\tilde{\B}_{L}, \C_{L}(\N))}
\end{diagram} \end{equation}
where the top map $\eta$ is the derived unit of the adjunction, the right-hand map is the equivalence of (\ref{eq:indec-triv}), the bottom isomorphism is the adjunction between the forgetful and cofree functors between symmetric sequences and $\K\E_L$-modules, and the left-hand map is adjoint to the composite
\[ \N \smsh_{\E_L} \tilde{\B}_{L} \weq \C_L(\N) \arrow{e,t}{\varepsilon} \N \]
where $\varepsilon$ is the counit map for the comonad $\C_L$. Since $\tilde{\B}_{L}$ is equivalent to $\one$ (as an $\E_L$-module in its first variable), this adjoint map is an equivalence. It is easy to check that the above diagram commutes and it follows that $\eta$ is an equivalence.

We use a similar method to prove that the derived counit $\epsilon: \Q\R\M \to \M$ is an equivalence for a bounded $\K\E_L$-module $\M$, this time reducing to the case of cofree $\K\E_L$-comodules. Suppose that $\M$ is $n$-truncated and consider the map of symmetric sequences
\[ \M \to \M_{=n}. \]
This is adjoint to a map of $\K\E_L$-modules
\[ \M \to \C_{L}(\M(n)) \]
which is an isomorphism in terms $k$ with $k \geq n$. (Those terms are trivial on both sides when $k > n$, and are both equal to $\M(n)$ when $k = n$.) There is therefore a fibre sequence of $\K\E_L$-modules of the form
\[ \M' \to \M \to \C_L(\M(n)) \]
where $\M'$ is $(n-1)$-truncated. We can now use induction on $n$ to reduce to proving that the derived counit is an equivalence for a bounded cofree $\K\E_L$-module, i.e. one of the form $\C_L(\A)$ for a bounded symmetric sequence $\A$. Note that this covers the base case of the induction because any $1$-truncated $\K\E_L$-module is cofree.

Now consider the following diagram (of symmetric sequences):
\begin{equation} \label{eq:koszul-counit} \begin{diagram}
  \node{\Map_{\Sigma}(\tilde{\B}_{L},\A) \smsh_{\E_L} \tilde{\B}_{L}} \arrow{e,t}{\sim} \arrow{s,l}{\isom}
    \node{\A \smsh_{\E_L} \tilde{\B}_{L}} \arrow{s,r}{\sim} \\
  \node{\Map_{\K\E_L}(\tilde{\B}_{L},\C_L(\A)) \smsh_{\E_L} \tilde{\B}_{L}} \arrow{e,t}{\epsilon}
    \node{\C_L(\A)}
\end{diagram} \end{equation}
analogous to (\ref{eq:koszul-unit}), where the bottom horizontal map is the derived counit of the adjunction and the right-hand vertical map is the equivalence of (\ref{eq:indec-triv}). It follows that $\epsilon$ is an equivalence for the cofree $\K\E_L$-module $\C_L(\A)$, and hence for all bounded $\K\E_L$-modules. This completes the proof that $(\Q,\R)$ restricts to an equivalence between the homotopy categories of bounded $\E_L$-comodules and bounded $\K\E_L$-modules.
\end{proof}

\begin{remark}
For an arbitrary operad $\P$ in $\spectra$, there is an adjunction of the form described in Proposition~\ref{prop:koszul-EL}, between $\P$-comodules and $\K\P$-modules. However, this is typically not an equivalence, even on bounded objects. The proof of \ref{prop:koszul-EL} relies heavily on \cite[3.67]{arone/ching:2014b} which in turn crucially depends on the fact that the operad term $\E_L(n)$ is a finite free $\Sigma_n$-spectrum (i.e. a cell $\Sigma_n$-spectrum formed from finitely many free cells). The corresponding result does hold for an operad $\P$ that shares this property.
\end{remark}

\subsection*{Pirashvili-type equivalence between $\E_L$-comodules and functors $\Gamma_L \to \spectra$}
We now turn to the equivalence between $\E_L$-comodules and functors $\Gamma_L \to \spectra$ that forms the bottom-left map in the diagram (\ref{eq:KE-mod-functor}). In this case, there is no need to restrict to bounded objects. This is a generalization of Theorem~3.78 of \cite{arone/ching:2014b} in the way that work of S{\l}omi\'{n}ska \cite{slominska:2004} generalizes that of Pirashvili \cite{pirashvili:2000}. It can be viewed as a covariant, and enriched, version of a theorem of Helmstutler \cite{helmstutler:2014}.

The categories $\Comod(\E_L)$ and $[\Gamma_L,\spectra]$ have projective model structures and in this section we build a Quillen equivalence between them. This is constructed, Morita-style in a similar manner to the adjunction of Proposition~\ref{prop:koszul-EL} but with a different `bimodule' object that we now introduce.

\begin{definition}
Define a functor
\[ \mathbb{B}_L : \un{\mathbb{E}}_L^{op} \times \Gamma_L \to \based \]
by
\begin{equation} \label{eq:pirashvili-bimod} \mathbb{B}_L(I,J_+) :=  \Wdge_{I \to J} \Smsh_{j \in J} \mathbb{E}_L(I_j). \end{equation}
Notice the similarity between this definition and that of $\un{\mathbb{E}}_L$ in (\ref{eq:PROP}). The difference is that the coproduct here is taken over \emph{all} functions from $I$ to $J$, not just the surjections. (Crucial here is that $\mathbb{E}_L(0)$ is non-trivial.) More precisely, we have an isomorphism
\begin{equation} \label{eq:BL-EL} \mathbb{B}_L(I,J_+) \isom \Wdge_{K \subseteq J} \un{\mathbb{E}}_L(I,K) \end{equation}
where a component on the left-hand side corresponding to a function $\alpha: I \to J$ is identified with a term on the right-hand side with $K = \alpha(I)$. The $\mathbb{E}_L$-module structure on the first variable of $\mathbb{B}_L$ is chosen so that this is an isomorphism of $\mathbb{E}_L$-modules.

There is also a close connection between $\mathbb{B}_L$ and the mapping spaces $\Gamma_L(I_+,J_+)$ for the category $\Gamma_L$. The latter object is a coproduct indexed by all pointed functions $\alpha: I_+ \to J_+$. We can think of $\mathbb{B}_L$ as consisting only of those components corresponding to $\alpha$ with $\alpha(I) \subseteq J$.

Here is another way to describe this relationship. For finite sets $I,J$, consider the cube of based spaces, indexed by subsets $K \subseteq I$, given by
\[ K \mapsto \Gamma_L(K_+,J_+). \]
The strict total cofibre of this cube (given, for example, by taking iterated cofibres in each direction) is then homeomorphic to $\mathbb{B}_L$, i.e. we have
\begin{equation} \label{eq:BL-GL} \tcofib_{K \subseteq I} \Gamma_L(K_+,J_+) \isom \mathbb{B}_L(I,J_+) \end{equation}
and we define the functoriality of $\mathbb{B}_L$ in its second variable in such a way that, for each $I$, this is a natural isomorphism of functors $\Gamma_L \to \based$. (The reader can check that the two structures commute yielding the required functor $\mathbb{B}_L$.
\end{definition}

The following calculation is crucial.

\begin{lemma} \label{lem:BL-thocofib}
For each finite set $I$, there is a natural weak equivalence, of functors $\Gamma_L \to \based$:
\[ \thocofib_{K \subseteq I} \Gamma_L(K_+,-) \weq \mathbb{B}_L(I,-). \]
\end{lemma}
\begin{proof}
Given (\ref{eq:BL-GL}), this amounts to showing that, for each $J_+ \in \Gamma_L$, the canonical map between the homotopy and strict total cofibres of the cube $\mathscr{X}(K) := \{\Gamma_L(K_+,J_+)\}$ is a weak equivalence. A standard condition for this is that, for each $K \subseteq I$, the map
\[ \colim_{L \subsetneq K} \mathscr{X}(L) \to \mathscr{X}(K) \]
is a cofibration of spaces. In our setting, that map is the inclusion into $\Gamma_L(K_+,J_+)$ of those components corresponding to pointed functions $\alpha: K_+ \to J_+$ for which $\alpha(K) \nsubseteq J$. An inclusion of components is a cofibration so this completes the proof.
\end{proof}

\begin{definition} \label{def:pirashvili}
We now build a Quillen equivalence of the form
\[ \L : \Comod(\E_L) \rightleftarrows [\Gamma_L,\spectra] : \X \]
that is constructed as follows.

Viewing $\mathbb{B}_L$ as a functor $\un{\mathbb{E}}_L^{op} \to [\Gamma_L,\based]$, we can compose with a (topologically-enriched) cofibrant replacement functor for $[\Gamma_L,\based]$ to obtain $\tilde{\mathbb{B}}_L : \un{\mathbb{E}}_L^{op} \times \Gamma_L \to \based$ such that $\tilde{\mathbb{B}}_L \homeq \mathbb{B}_L$ and with the property that for each finite set $I$, $\tilde{\mathbb{B}}_L(I,-)$ is cofibrant in $[\Gamma_L,\based]$.

Now define $\L: \Comod(\E_L) \to [\Gamma_L,\spectra]$ by
\[ \L(\N) := \N \smsh_{\E_L} \tilde{\mathbb{B}}_L \]
with right adjoint $\X: [\Gamma_L,\spectra] \to \Comod(\E_L)$ given by
\[ \X(G) := \Map_{\Gamma_L}(\tilde{\mathbb{B}}_L,G), \]
a coend and end, respectively, using, respectively, the tensoring and cotensoring of $\spectra$ over $\based$.
\end{definition}

\begin{proposition} \label{prop:pirashvili}
The adjunction of Definition~\ref{def:pirashvili} is a Quillen equivalence between the projective model structures on $\Comod(\E_L)$ and $[\Gamma_L,\spectra]$.
\end{proposition}
\begin{proof}
This proof is very similar to the proof in \cite{arone/ching:2014b} that the bottom-left map in (\ref{eq:classification}) is part of a Quillen equivalence. First note that for each finite nonempty set $I$, the object $\tilde{\mathbb{B}}_L(I,-)$ is a cofibrant object in the projective model structure on $[\Gamma_L,\spectra]$. It follows that the right adjoint $\X$ preserves fibrations and trivial fibrations (since these are detected objectwise in both categories) and hence that $(\L,\X)$ is a Quillen adjunction.

Our goal is now to show that the derived unit and counit of this adjunction are equivalences. Key to this is the equivalence of Lemma~\ref{lem:BL-thocofib}. Applying the derived mapping space construction $\widetilde{\Map}_{J_+ \in \Gamma_L}(-, G(J_+))$ to this, we get, using Yoneda, an equivalence
\begin{equation} \label{eq:R-creff} \X(G)(I) \weq \thofib_{K \subseteq I} G(K_+). \end{equation}
This describes $\X(G)(I)$ as the \ord{$I$} cross-effect of $G$ evaluated at $(D^L_+,\dots,D^L_+)$, hence the notation $\X$.

In particular, it follows from this calculation that the right adjoint $\X$ preserves homotopy colimits, since the cross-effect can also be recovered from the total homotopy cofibre of the same cube. We can also use (\ref{eq:R-creff}) to show, by induction, that $\X$ reflects equivalences, i.e. if a natural transformation $g: G \to G'$ induces an equivalence of $\E_L$-comodules $\X(G) \weq \X(G')$, then $g$ itself is an equivalence.

Now consider the derived unit map $\N \to \X\L\N$ for a $\E_L$-comodule $\N$. Since both $\X$ and $\L$ commute with homotopy colimits, we can use the cofibrantly generated model structure on $\Comod(\E_L)$ to reduce to the case that $\N$ is either the source or target of one of the generating cofibrations, i.e. of the form
\[ \N(r) = A \smsh \un{\E}_L(n,r) \]
for some finite spectrum $A$ and positive integer $n$. In this case, the unit map reduces to
\[ A \smsh \un{\E}_L(n,r) \to \Map_{J_+ \in \Gamma_L}(\tilde{\mathbb{B}}_L(r,J_+),A \smsh \tilde{\mathbb{B}}_L(n,J_+)) \]
and so, since $A$ is finite, it is sufficient to show that the canonical map
\begin{equation} \label{eq:EL} \un{\E}_L(n,r) \to \Map_{J_+ \in \Gamma_L}(\Sigma^\infty \tilde{\mathbb{B}}_L(r,J_+), \Sigma^\infty \tilde{\mathbb{B}}_L(n,J_+)) \end{equation}
(determined by the $\mathbb{E}_L$-module structure on $\mathbb{B}_L$) is a weak equivalence. This follows from the existence of the following commutative diagram:
\[ \begin{diagram}
  \node{\un{\E}_L(n,r)} \arrow{e} \arrow{s,l}{\sim} \node{\Map_{J_+ \in \Gamma_L}(\Sigma^\infty \tilde{\mathbb{B}}_L(r,J_+), \Sigma^\infty \tilde{\mathbb{B}}_L(n,J_+))} \arrow{s,r}{\sim} \\
  \node{\thofib_{K \subseteq \un{r}} \Wdge_{I \subseteq K} \un{\E}_L(n,I)} \arrow{e,t}{\sim} \node{\thofib_{K \subseteq \un{r}} \Sigma^\infty \mathbb{B}_L(n,K_+)}
\end{diagram} \]
where the right-hand map is the equivalence of (\ref{eq:R-creff}), the bottom map comes from (\ref{eq:BL-EL}), and the left-hand vertical map is induced by the inclusion of $\un{\E}_L(n,r)$ into the initial vertex of the given cube as the term where $I = K = \un{r}$. It is a simple calculation to show that this map is an equivalence. This completes the proof that the derived unit map is an equivalence for all $\N$.

Finally, we turn to the derived counit map $\epsilon_G: \L\X G \to G$. First note that the unit map
\[ \X G \to \X\L\X G \]
is an equivalence by the above, and so the triangle identity implies that
\[ \X(\epsilon_G) : \X\L\X G \to \X G \]
is an equivalence. But we have already shown that $\X$ reflects equivalences, so we deduce that $\epsilon_G$ is an equivalence too. This completes the proof that $(\L,\X)$ is a Quillen equivalence.
\end{proof}

\begin{proof}[Proof of Theorem~\ref{thm:KE-mod-functor}]
We now return to diagram (\ref{eq:KE-mod-functor}). We construct this as follows:
\begin{equation} \label{eq:KE-mod-functor-big} \begin{diagram} \dgARROWLENGTH=3em
  \node[2]{\Mod(\K\E_L)} \arrow[2]{e,t}{\U}
    \node[2]{\Coalg(\C)} \arrow{se,t}{\A \mapsto F_{\A}} \\
  \node{\Comod(\E_L)} \arrow{ne,t}{\Q} \arrow[2]{e,t}{- \homsmsh_{\E_L} \un{\Com}} \arrow{se,b}{\L}
    \node[2]{\Comod(\Com)} \arrow{ne,t}{\hocolim_L - \homsmsh_{\E_L} \tilde{\B}_{L}}
    \arrow[2]{e,t}{- \smsh_{\Com} \widetilde{\Sigma^\infty X^{\smsh *}}}
    \arrow{se,b}{- \smsh_{\Com} \widetilde{\Sigma^\infty (J_+)^{\smsh *}}}
    \node[2]{[\finbased,\spectra]} \\
  \node[2]{[\Gamma_L,\spectra]} \arrow[2]{e,b}{\mathsf{hLKan}}
    \node[2]{[\Gamma,\spectra]} \arrow{ne,b}{\mathsf{hLKan}}
\end{diagram} \end{equation}
The right-hand part of this diagram is (\ref{eq:classification}) without the restrictions to bounded objects or polynomial functors. The bottom horizontal map is (enriched) homotopy left Kan extension along the functor $\pi_0: \Gamma_L \to \Gamma$ that sends the object $\Wdge_I D^L_+$ of $\Gamma_L$ to the finite pointed set $I_+$, and which forgets the `topological' information about a morphism. To finish the proof of Theorem~\ref{thm:KE-mod-functor}, it is now sufficient to show that the two quadrilaterals in this diagram commute up to natural equivalence.

For the top-left quadrilateral, take an $\E_L$-comodule $\N$. The composite of the middle horizontal and diagonal arrows applied to $\N$ gives
\[ \hocolim_{L'} (\N \homsmsh_{\E_L} \un{\Com} \homsmsh_{\E_{L'}} \tilde{\B}_{L'} \]
which is equivalent to
\[ \N \homsmsh_{\E_L} (\hocolim_{L'} \un{\Com} \homsmsh_{\E_{L'}} \tilde{\B}_{L'}). \]
Writing $\un{\Com} = \hocolim_{L'} \un{\E}_{L'}$ this reduces to
\[ \tag{*} \N \homsmsh_{\E_L} (\hocolim_{L'} \tilde{\B}_{L'}). \]
The maps
\[ \B_0 \to \B_1 \to \dots \]
induced by the operad maps $\E_0 \to \E_1 \to \dots$ are all equivalences (each term is equivalent to the trivial bisymmetric sequence $\un{\one}$) and so in particular we have an equivalence (of $\E_L$-modules in one variable and $\C$-coalgebras in the other):
\[ \hocolim_{L'} \tilde{\B}_{L'} \homeq \tilde{\B}_L \]
so (*) is equivalent to $\N \homsmsh_{\E_L} \tilde{\B}_L = \U(\Q(\N))$.

For the bottom-left quadrilateral, it is sufficient to show the diagram formed by replacing the horizontal functors with their right adjoints commutes up to natural equivalence. This diagram has the form
\[ \begin{diagram}
  \node{\Comod(\E_L)} \arrow{s,lr}{- \smsh_{\E_L} \tilde{\mathbb{B}}_L}{\sim}
    \node{\Comod(\Com)} \arrow{s,lr}{\sim}{- \smsh_{\Com} \widetilde{\Sigma^\infty (J_+)^{\smsh *}}} \arrow{w,t}{\U} \\
  \node{[\Gamma_L,\spectra]} \node{[\Gamma,\spectra]} \arrow{w,b}{\mathsf{res}}
\end{diagram} \]
Each direction here takes the form
\[ \N \mapsto  \left[ J_+ \mapsto \Wdge_{K \subseteq J} \N(K) \right] \]
which proves the claim. This completes the construction of the diagram (\ref{eq:KE-mod-functor-big}). Theorem~\ref{thm:KE-mod-functor} now follows from the fact that $\Q$ and $\L$ are equivalences ($\Q$ only on those bounded objects that correspond to polynomial functors).
\end{proof}

\begin{example}
Taking $L = 0$ in Theorem~\ref{thm:KE-mod-functor} we see that the following conditions are equivalent for a polynomial functor $F: \finbased \to \spectra$:
\begin{itemize}
  \item the $\C$-coalgebra structure on $\der_*(F)$ is `trivial' (in the sense that it is equivalent to that arising from the symmetric sequence $\der_*(F)$ via the coaugmentation $1 = \C_0 \to \C$);
  \item the Taylor tower of $F$ splits;
  \item $F$ is the left Kan extension from a functor on the subcategory of $\finbased$ consisting of the finite pointed sets and functions $f: I_+ \to J_+$ that are injections away from the basepoint (i.e. for each $j \in J$, $f^{-1}(j)$ has at most one element).
\end{itemize}
\end{example}

\section{A category of pointed framed manifolds} \label{sec:fMfld}

In this section, we use Theorem~\ref{thm:KE-mod-functor} to find a wider variety of functors whose derivatives are a $\K\E_L$-module. In particular, we consider the question of when the derivatives of a representable functor, i.e. one of the form $\Sigma^\infty \Hom_{\based}(X,-)$ for a finite based CW-complex $X$, have the structure of a $\K\E_L$-module for some given $L$.

We show that this is the case when $X$ is (or is homotopy equivalent to) a `pointed framed $L$-dimensional manifold' in the sense of Definition~\ref{def:pfm} below. This means that when the basepoint of $X$ is removed, the remaining space can be given the structure of a framed smooth manifold of dimension $L$, possibly with boundary. In particular, notice that the sphere $S^L$ has this property, so we deduce that the iterated (stable) loop-space functor $\Sigma^\infty \Omega^L$ has a $\K\E_L$-module structure on its derivatives. Similarly, since $S^1_+$ is a pointed framed $1$-manifold, the derivatives of the free loop-space functor $\Sigma^\infty L = \Sigma^\infty \Hom_{\based}(S^1_+,-)$ form a $\K\E_1$-module.

We also give conditions under which a map $f:X \to Y$ between pointed framed manifolds induces a map between the derivatives of the representable functors that is a morphism of $\K\E_L$-modules. This happens when $f$ is a `pointed framed embedding', i.e. $f$ is an embedding (when restricted to the part of $X$ that does not map to the basepoint in $Y$) that respects the framing.

All the manifolds we consider in this section are smooth, possibly with boundary. A \emph{framing} on a smooth $L$-dimensional manifold $M$ is a choice of isomorphism of vector bundles
\[ TM \isom M \times \mathbb{R}^L. \]
In particular we do not mean `stably-framed', that is, we do not allow for the addition of trivial bundles to $TM$ get such an isomorphism. Such a structure therefore exists if and only if $M$ is parallelizable.

Given a smooth map $f: M \to N$ between smooth $L$-manifolds, each with a framing, we can express the derivative $Df(x) : T_xM \to T_{f(x)}N$ as an $(L \times L)$ matrix via the identifications $T_xM \isom \mathbb{R}^L$ and $T_{f(y)}N \isom \mathbb{R}^L$ that come with the framings. We say that the map $f$ is \emph{framed} if there is a locally constant function $\lambda: M \to \mathbb{R}^+$ such that $Df(x) = \lambda(x) \mathbf{Id}$ for all $x \in M$. This means that, on each component of $M$, the derivative is a constant positive scalar multiple of the identity matrix.

\begin{definition} \label{def:pfm}
Fix an integer $L \geq 0$. A \emph{pointed framed $L$-manifold} is a based topological space $(X,x_0)$ together with the structure of a framed smooth $L$-dimensional manifold, possibly with boundary, on the topological space $X - \{x_0\}$.

Given two pointed framed $L$-manifolds $(X,x_0)$ and $(Y,y_0)$, a \emph{pointed framed embedding} $f: X \to Y$ is a continuous basepoint-preserving map such that the restriction of $f$ to $f^{-1}(Y - \{y_0\})$ is a framed embedding. We write $\Emb^f_*(X,Y)$ for the space of pointed framed embeddings $X \to Y$ with the compact-open topology. These spaces form the mapping spaces in a topologically-enriched category of pointed framed $L$-manifolds.
\end{definition}

\begin{remark}
Our pointed framed manifolds are a smooth and framed version of the `zero-pointed manifolds' developed independently, and much more extensively, by Ayala and Francis \cite{ayala/francis:2014} for their study of factorization homology.
\end{remark}

We now introduce a condition on a pointed framed manifold $(X,x_0)$ that restricts how the topology around the basepoint $x_0$ is related to the framing on $X - \{x_0\}$. The idea is that $x_0$ should have a basis of open neighbourhoods that are contractible in a framed sense.

\begin{definition} \label{def:f-contr}
A pointed framed manifold $(V,x_0)$ is \emph{f-contractible} if there exists a homotopy
\[ h: V \smsh [0,1]_+ \to V \]
such that $h_1 = \mathsf{id}_V$, $h_0$ is the constant map with value $x_0$, and, for each $t \in [0,1]$, $h_t: V \to V$ is a pointed framed embedding. Thus $V$ deformation retracts onto $x_0$ through pointed framed embeddings.

A pointed framed manifold $(X,x_0)$ is \emph{locally f-contractible} if there is a basis of open neighbourhoods $\{V_\alpha\}$ around $x_0$ such that each $(V_{\alpha},x_0)$ is framed contractible.
\end{definition}

\begin{examples} \label{ex:fmfld}
Let $M$ be a framed $L$-dimensional manifold with boundary. Then we can get locally f-contractible pointed framed manifolds from $M$ in a variety of ways:
\begin{enumerate}
  \item For any $x \in M$, $(M,x)$ is a locally f-contractible pointed framed manifold (with the induced framing on $M - \{x\}$. In this case, a basis of f-contractible neighbourhoods of $x$ is given by the framed embedded discs with centre $x$.
  \item Adding a disjoint basepoint $+$, we get a locally f-contractible pointed framed manifold $(M_+,+)$.
  \item We can add a `framed collar' to the boundary of $M$ by forming $M' := (M \times \{0\}) \cup_{\partial M \times \{0\}} (\partial M \times [0,1])$ where $\partial M \times [0,1]$ has the product framing. The quotient $M'/(\partial M \times \{1\})$ is then a locally f-contractible pointed framed manifold that is homotopy equivalent to $M/\partial M$.
\end{enumerate}
\end{examples}

\begin{examples}
The sphere $S^L$ can be given the structure of a locally f-contractible pointed framed $L$-manifold by thinking of it as the one-point compactification, either of $\mathbb{R}^L$ or of the open unit disc $\mathring{D}^L$ inside $\mathbb{R}^L$. In each case, a basis of f-contractible neighbourhoods of the basepoint consists of the complements of the closed discs centred the origin, with contractions given by Euclidean dilations.
\end{examples}

\begin{definition}
We now introduce a full subcategory of the category of pointed framed $L$-manifolds that we denote $\fMfld^L_*$. The objects of $\fMfld^L_*$ are those locally f-contractible pointed framed manifolds whose underlying space has the structure of a finite cell complex. In particular, we have a forgetful functor
\[ U: \fMfld^L_* \to \finbased. \]
\end{definition}

\begin{remark}
Any finite cell complex $X$ is homotopy equivalent to a compact codimension zero submanifold of Euclidean space, and hence to a compact framed manifold. Thus any based finite cell complex $X$ is homotopy equivalent to an object of $\fMfld^L_*$ for some $L$. Similarly, any map $f:X \to Y$ between finite cell complexes can be modelled as a pointed framed embedding by embedding the mapping cylinder of $f$ into some Euclidean space.
\end{remark}

\begin{examples}
The category $\fMfld^0_*$ is equivalent to that denoted $\Gamma_0$ in Definition~\ref{def:Gamma-L}: the objects are finite pointed sets, and the morphisms are those basepoint preserving functions that are injective away from the basepoint.

For $L > 0$, the category $\Gamma_L$ is the full subcategory of $\fMfld^L_*$ whose objects are the finite wedge sums $\Wdge_n D^L_+$. Interestingly, $\fMfld^L_*$ also contains a full subcategory that is equivalent to $\Gamma_L^{op}$, namely that whose objects are the finite wedge sums $\Wdge_n S^L$ (with framing coming from that on the open unit disc). The morphisms in this case are the Thom-Pontryagin collapse maps associated to the corresponding morphisms in $\Gamma_L$.
\end{examples}

The following lemma contains the key technical result of this section: a calculation of the homotopy type of the space of pointed framed embeddings from a finite wedge sum $\Wdge_n D^L_+$ to a locally f-contractible pointed framed manifold $X$.

\begin{lemma} \label{lem:emb}
Let $X$ be a locally f-contractible pointed framed manifold, and let $C_1(n,X)$ denote the subspace of $X^n$ consisting of $n$-tuples $(x_1,\dots,x_n)$ where $x_i \neq x_j$ for $i \neq j$, unless $x_i = x_j = x_0$. (Thus $C_1(n,X)$ is the ordinary configuration space of $n$ points in $X$ with the exception that repetitions of the basepoint are allowed.) Then there is a natural weak equivalence
\[ \Theta: \Emb^f_*\left(\Wdge_n D^L_+,X\right) \weq C_1(n,X) \]
given by
\[ \Theta(f) := (f_1(0),\dots,f_n(0)). \]
\end{lemma}
\begin{proof}
We use the following criterion of McCord~\cite[Thm. 6]{mccord:1966} to show that $\Theta$ is a weak equivalence.

\begin{lemma}[McCord] \label{lem:mccord}
Let $f: Y \to Z$ be a continuous map and let $\mathcal{U}$ be a basis for the topology on $Z$ such that, for each $U \in \mathcal{U}$ the map $f^{-1}(U) \to U$ is a weak equivalence. Then $f$ is a weak equivalence.
\end{lemma}

We construct the necessary basis $\mathcal{U}$ as follows. Let $\mathcal{V}$ denote a basis of f-contractible open neighbourhoods of $x_0$. The elements of $\mathcal{U}$ are then the open subsets of $C_1(n,X)$ of the form
\[ U = [V_1 \times \dots \times V_n] \cap C_1(n,X) \]
where each $V_i$ is either: (i) an embedded open disc in $X - \{x_0\}$ (whose closure is also contained in $X - \{x_0\}$); or (ii) an element of $\mathcal{V}$.

We insist that $V_1,\dots,V_n$ are pairwise disjoint with the exception of those that are in $\mathcal{V}$ which we insist are equal. In other words, $\{V_1,\dots,V_n\}$ is a collection of disjoint embedded open discs in $X - \{x_0\}$, possibly together with an element of $\mathcal{V}$ (that can be repeated). The corresponding set $U$ then consists of the $n$-tuples $(x_1,\dots,x_n)$ where $x_i \in V_i$, and the $x_i$ are distinct (though repetitions of the basepoint are permitted).

It is not hard to see that $\mathcal{U}$ is a basis for $C_1(n,X)$. Each $U \in \mathcal{U}$ is contractible: there is a deformation retraction from $U$ to the point $(x_1,\dots,x_n)$ where $x_i$ is the center of each embedded disc $V_i$, and $x_i = x_0$ for those $V_i$ that equal $V \in \mathcal{V}$. A deformation retraction is given by combining a contraction of each embedded disc to its center with a contraction $c$ of $V$ of the form guaranteed by Definition~\ref{def:f-contr}.

It is now sufficient by Lemma~\ref{lem:mccord} to show that each $\Theta^{-1}(U)$ is contractible. We do this in two stages:
\begin{enumerate}
  \item construct a homotopy between the identity on $\Theta^{-1}(U)$ and a map whose image is contained in a certain subset $\Delta(U) \subseteq \Theta^{-1}(U)$;
  \item show that $\Delta(U)$ is contractible.
\end{enumerate}

Notice that $\Theta^{-1}(U)$ is the set of pointed framed embeddings
\[ f: \Wdge_n D^L_+ \to X \]
such that $f_i(0) \in V_i$ for each $i$. We take the subset $\Delta(U)$ to consist of those $f$ such that $f_i(D^L) \subseteq V_i$ for each $i$.

Our strategy for part (1) of the proof is simply to `shrink' a given embedding $f$ until its image discs are contained in the sets $V_i$ as required. The Tube Lemma ensures that this can be done for some nonzero scale factor. More precisely, we set, for $f \in \Theta^{-1}(U)$:
\[ s(f) := \sup\{t \in [0,1] \; | \; \text{$f_i(ty) \in V_i$ for any $y \in D^L$ and $i = 1,\dots,n$} \}. \]
The number $s(f) \in [0,1]$ is well-defined since $f_i(0) \in V_i$ for each $i$ and the Tube Lemma implies that $s(f) > 0$. We use the construction of $s(f)$ to define a homotopy
\[ h: \Theta^{-1}(U) \smsh [0,1]_+ \to \Theta^{-1}(U) \]
by
\[ h(f,t)_i(y) := \begin{cases}
  f_i(ty) & \text{if $t \geq s(f)/2$}; \\
  f_i\left(\frac{s(f)}{2}y\right) & \text{if $t < s(f)/2$}.
\end{cases} \]
Then $h(f,1) = f$ and $h(f,0) \in \Delta(U)$ for each $f \in \Theta^{-1}(U)$. For part (1) of the proof we just need to show that $h$ is continuous, for which it is sufficient to show that $s: \Theta^{-1}(U) \to (0,1]$ is a continuous function. To prove this, we have to understand the topology on $\Theta^{-1}(U)$.

Recall the definition of the compact-open topology on $\Emb^f_*(\Wdge_n D^L_+,X)$. A basis for this topology is given by sets of the form
\[ O(K_1,W_1) \cap \dots \cap O(K_r,W_r) \]
for compact $K_1,\dots,K_r \subseteq \Wdge_n D^L_+$ and open $W_1,\dots,W_r \subseteq X$, where $O(K,W)$ consists of those $f$ such that $f(K) \subseteq W$.

Now suppose that $0 < s(f) < 1$ and take $\epsilon > 0$ such that $(s(f)-\epsilon,s(f)+\epsilon) \subseteq [0,1]$. First, note that by definition of $s(f)$, we have, for all $i$,
\[ f_i([s(f)-\epsilon]D^L) \subseteq V_i. \]
and also that there is some $j$ and some $y \in D^L$ such that
\[ x' := f_j([s(f)+\epsilon/2]y) \notin V_j. \]
Suppose that $x' \in \bar{V}_j$. Then the embedding $f_j$ must embed some disc around $[s(f)+\epsilon/2]y$ as some disc around $x'$. Since $x' \in \bar{V}_j - V_j$, some point in this disc must be outside of $\bar{V}_j$. We can therefore find some $y' \in D^L$ and some $0 < \epsilon' < \epsilon$ such that
\[ x'' := f_j([s(f)+\epsilon']y') \]
is contained in an embedded disc $D$ that does not intersect $V_j$.

We now observe that $f$ is contained in the open set
\[ O(\{[s(f)+\epsilon']y'\}_{(j)}, D) \cap \bigcap_{i = 1}^{n} O([s(f)-\epsilon]D^L_{(i)}, V_i). \]
Moreover, if $g$ is also in this open set, then
\[ s(f) - \epsilon < s(g) < s(f) + \epsilon. \]
This shows that $s$ is continuous at $f$.

Similarly, if $s(f) = 1$ and $0 < \epsilon < 1$, then there is some open subset of $\Emb^f_*(\Wdge_n D^L_+,X)$ containing $f$ and contained in $s^{-1}((1-\epsilon,1])$. This concludes the proof that $s$ is continuous, and it follows that $h$ is continuous. This completes part (1).

Now, for part (2) of the proof, we show that the space $\Delta(U)$ is contractible. Let $I = \{i \in \un{n} \; | \; V_i = V \in \mathcal{V}\}$. We then have
\[ \Delta(U) \isom \Emb^f_*(\Wdge_{I} D^L_+,V) \times \prod_{i \notin I} \Emb^f(D^L,V_i) \]
where $\Emb^f(D^L,V_i)$ is the space of framed embeddings $D^L \to V_i$. Recall that $V_i$ is an embedded open disc in $X - \{x_0\}$ with some centre $x_i$ and some finite radius $r_i > 0$. This space of embeddings is clearly contractible. (For example, it deformation retracts onto the single embedding with centre $x_i$ and radius $r_i/2$.)

Finally, if $c: V \smsh [0,1]_+ \to V$ is a contraction of $V$ to $x_0$ such that $c_t : V \to V$ is a pointed framed embedding for each $t$, then we define a contraction $c'$ of $\Emb^f_*(\Wdge_{I} D^L_+,V)$ by
\[ c'_t(f) := c_t \circ f. \]
Altogether then, we deduce that $\Delta(U)$ is contractible, which completes part (2).

Combining a contraction of $\Delta(U)$ with the homotopy $h$ defined above, we obtain a contraction of $\Theta^{-1}(U)$ as required. This completes the proof of Lemma~\ref{lem:emb}.
\end{proof}

\begin{remark}
For us, the real significance of Lemma~\ref{lem:emb} is that the configuration spaces $C_i(n,X)$ satisfy a cosheaf property with respect to the covering of $X$ by open subsets that are a disjoint union of embedded discs and an f-contractible neighbourhood of the basepoint. Lemma~\ref{lem:emb} then implies that the embedding spaces $\Emb^f_*(\Wdge_n D^L_+,X)$ satisfy the same property. We use this property in the proof of Proposition~\ref{prop:MX} below.
\end{remark}

We now want to construct a $\K\E_L$-module that models the derivatives of the representable functor $\Sigma^\infty \Hom_{\based}(X,-)$ when $X \in \fMfld^L_*$. It is easier to describe first an $\E_L$-module associated to $X$, from which we get the required $\K\E_L$-module via a form of Koszul duality.

\begin{definition}
We define a functor $\mathbb{M}_\bullet: \fMfld^L_* \to \Mod(\mathbb{E}_L)$ by
\[ \mathbb{M}_X(n) := \mathbb{B}_L(n,J_+) \homsmsh_{J_+ \in \Gamma_L} \Emb^f_*\left( \Wdge_J D^L_+, X \right) \]
where $\mathbb{B}_L$ is as in (\ref{eq:pirashvili-bimod}) and $\homsmsh$ denotes the (derived) enriched homotopy coend over the category $\Gamma_L$ of Definition~\ref{def:Gamma-L}. The symmetric sequence $\mathbb{M}_X$ inherits an $\mathbb{E}_L$-module structure from that on the first variable of $\mathbb{B}_L$, and the functoriality in $X$ arises from that of the embedding space functor $\Emb^f_*(-,-)$.
\end{definition}

\begin{example} \label{ex:MDL}
When $X = D^L_+$ we have an equivalence of $\mathbb{E}_L$-modules
\[ \mathbb{M}_X \homeq \mathbb{B}_L(*,1_+) \isom \mathbb{E}_L. \]
\end{example}

\begin{remark} \label{rem:MX-config}
In Example~\ref{ex:MDL}, we see that $\mathbb{M}_X(n)$ has the homotopy type of a configuration space of points in $X$. This is a general phenomenon. It can be shown that there is a natural weak equivalence
\[ \eta_X: \mathbb{M}_X(n) \weq C_*(n,X) \]
where $C_*(n,X)$ denotes the subspace of $X^{\smsh n}$ consisting of those $n$-tuples of distinct points in $X$, together with the basepoint. The map $\eta_X$ is constructed from the maps
\[ \mathbb{B}_L(n,J_+) \smsh \Emb^f_*\left(\Wdge_J D^L_+,X\right) \to C_*(n,X); \quad (f,g) \mapsto (g(f_1(0)),\dots,g(f_n(0))) \]
where $f \in \mathbb{B}_L(n,J_+)$ is identified with a sequence of embeddings $f_i: D^L \to \coprod_J D^L$.
\end{remark}

\begin{example} \label{ex:SL}
For $L > 0$ and $X = S^L$ we have
\[ \mathbb{M}_{S^L}(n) \homeq \begin{cases}
  S^L & \text{for $n = 1$}; \\
  * & \text{for $n > 1$}.
\end{cases} \]
This follows from the equivalence of Remark~\ref{rem:MX-config}. The case $n = 1$ is immediate and for $n > 1$ it suffices to show that $C_*(n,S^L)$ is contractible. Such a contraction is given by $c: C_*(n,S^L) \smsh [0,1]_+ \to C_*(n,S^L)$ by
\[ c([x_1,\dots,x_n],t) := \begin{cases} [\frac{1}{t}x_1,\dots,\frac{1}{t}x_n] & \text{for $t > 0$}; \\ \quad * & \text{for $t = 0$}; \end{cases} \]
where we write $S^L = \mathbb{R}^L/\{x : |x| \geq 1\}$. For any $n$-tuple $x_1,\dots,x_n$ in $S^L$, at least one of the points $x_i$ is not represented by $0 \in \mathbb{R}^L$ and hence $\frac{1}{t}x_i$ represents the basepoint in $S^L$ for sufficiently small $t$. Thus $c$ is continuous and provides the necessary contraction.
\end{example}

\begin{remark}
For a framed compact $L$-manifold $M$ (with boundary), the based spaces $M_+$ and $M/\partial M$ are connected by Atiyah duality. This appears to correspond to a version of Koszul duality between the $\E_L$-modules $\mathbb{M}_{M_+}$ and $\mathbb{M}_{M/\partial M}$. (Recall from Example~\ref{ex:fmfld} that we can build a locally f-contractible pointed framed manifold that is homotopy equivalent to $M/\partial M$.) For example, notice from Examples~\ref{ex:MDL} and \ref{ex:SL} that $\mathbb{M}_{D^L_+}$ and $\mathbb{M}_{S^L}$ are the free and trivial $\mathbb{E}_L$-modules on one `generator' respectively. This is a module-level (as opposed to algebra-level) version of an observation by Ayala and Francis relating Poincar\'{e} and Koszul duality via factorization homology.
\end{remark}

The key property of the $\mathbb{E}_L$-module $\mathbb{M}_X$ is now given by the following result. Let $\mathbbm{Com}$ denote the commutative operad in based spaces: with $\mathbbm{Com}(n) = S^0$ for all $n$. The operad map $\mathbb{E}_L \to \mathbbm{Com}$ makes $\mathbbm{Com}$ into a left $\mathbb{E}_L$-module and we can then form the two-sided bar construction $\B(\mathbb{M}_X,\mathbb{E}_L,\mathbbm{Com})$.

\begin{proposition} \label{prop:MX}
There is an equivalence of $\mathbbm{Com}$-modules, natural in $X \in \fMfld^L_*$ of the form
\[ \theta_X: \B(\mathbb{M}_X,\mathbb{E}_L,\mathbbm{Com}) \weq X^{\smsh *}. \]
\end{proposition}
\begin{proof}
The map $\theta_X$ is built from a map $\mathbb{M}_X \circ \mathbbm{Com} \to X^{\smsh *}$ which has components, associated to a surjection $\alpha: \un{n} \epi \un{k}$, of the form
\[ \mathbb{M}_X(k) \to X^{\smsh n} \]
given by the composite
\[ \mathbb{M}_X(k) \arrow{e,t}{\eta_X} C_*(k,X) \to X^{\smsh k} \arrow{e,t}{\Delta_{\alpha}} X^{\smsh n} \]
where $\eta_X$ is as in Remark~\ref{rem:MX-config}.

To prove that $\theta_X$ is an equivalence, we start with the case $X = \Wdge_I D^L_+$, for some finite set $I$. We then have
\[ \B(\mathbb{M}_X,\mathbb{E}_L,\mathbbm{Com})(n) \homeq \B(\Wdge_{K \subseteq I} \un{\mathbb{E}}_L(n,K), \mathbb{E}_L,\mathbbm{Com}) \isom \Wdge_{K \subseteq I} \un{\mathbbm{Com}}(n,K). \]
We can identify $\un{\mathbbm{Com}}(n,K)$ with the subset of $(I_+)^{\smsh n}$ consisting of $n$-tuples in $I$ whose union is $K$. Taking the sum over all subsets $K$ of $I$, we get an equivalence
\[ \B(\mathbb{M}_X,\mathbb{E}_L,\mathbbm{Com})(n) \weq (I_+)^{\smsh n}. \]
The map $\theta_X$ factors this via the equivalence
\[ X^{\smsh n} \to (I_+)^{\smsh n} \]
given by contracting each component of $X$ to a single point. It follows that $\theta_X$ is an equivalence in this case.

Next consider the case that $X = \Wdge_I D^L_+ \wdge K$ for some Hausdorff f-contractible pointed framed manifold $K$. (So $X$ consists of $K$ together with a disjoint union of discs.) It follows that $X$ is homotopy equivalent to $X' :=\Wdge_I D^L_+$ (with the relevant homotopies consisting of pointed framed embeddings). This equivalence induces corresponding equivalences of $\mathbbm{Com}$-modules
\[ \B(\mathbb{M}_X,\mathbb{E}_L,\mathbbm{Com})(n) \homeq \B(\mathbb{M}_{X'},\mathbb{E}_L,\mathbbm{Com})(n); \quad X^{\smsh n} \homeq X'^{\smsh n} \]
and it follows by the previous case that $\theta_{X}$ is an equivalence.

For an arbitrary $X \in \fMfld^L_*$ we now let $\mathcal{P}$ denote the poset of subsets of $X$ that are of the form (i.e. isomorphic as pointed framed manifolds) $\Wdge_I D^L_+ \wdge K$ for f-contractible $K$. Consider the following commutative diagram:
\begin{equation} \label{eq:diag-hocolim} \begin{diagram}
  \node{\hocolim_{U \in \mathcal{P}} \B(\mathbb{M}_U,\mathbb{E}_L,\mathbbm{Com})} \arrow{e} \arrow{s,l}{\sim} \node{\B(\mathbb{M}_X,\mathbb{E}_L,\mathbbm{Com})} \arrow{s,r}{\theta_X} \\
  \node{\hocolim_{U \in \mathcal{P}} U^{\smsh *}} \arrow{e} \node{X^{\smsh *}}
\end{diagram} \end{equation}
Since the left-hand vertical map is an equivalence by the previous case, it is sufficient to show that the horizontal maps are weak equivalences.

The bar construction $\B(\M_X,\E_L,\Com)$ is built from the embedding functor $\Emb^f_*(\Wdge_n D^L_+,-)$ by taking various homotopy colimits and so it is sufficient to prove that
\[ \hocolim_{U \in \mathcal{P}} \Emb^f_*(\Wdge_n D^L_+,U) \to \Emb^f_*(\Wdge_n D^L_+,X) \]
is an equivalence, for which, by Lemma~\ref{lem:emb}, it is enough to show that
\begin{equation} \label{eq:hocolim-C1} \hocolim_{U \in \mathcal{P}} C_1(n,U) \to C_1(n,X) \end{equation}
is an equivalence for any $n \geq 0$.

We prove this using the following result of Dugger-Isaksen~\cite[1.6]{dugger/isaksen:2004}:

\begin{lemma}[Dugger/Isaksen] \label{lem:dugger/isaksen}
Let $\mathscr{U}$ be an open cover of a space $\mathcal{X}$ such that each finite intersection of elements of $\mathscr{U}$ is covered by other elements of $\mathscr{U}$. Then the homotopy colimit of the diagram formed by the elements of $\mathscr{U}$ and the inclusions between them is weakly equivalent to $\mathcal{X}$.
\end{lemma}

For $U \in \mathcal{P}$, let $\mathring{U}$ denote the interior of $U$ in $X$, i.e. the disjoint union of a collection of open discs in $X$ with an f-contractible neighbourhood of the basepoint. We apply Lemma~\ref{lem:dugger/isaksen} to the open cover
\[ \mathscr{U} := \{C_1(n,\mathring{U}) \subseteq C_1(n,X) \; | \; U \in \mathcal{P}\}. \]
This collection covers $C_1(n,X)$ because each $n$-tuple of points in $X$ is contained in some $U \in \mathcal{P}$. Consider a point $x = (x_1,\dots,x_n)$ in the finite intersection:
\[ C_1(n,\mathring{U}_1) \cap \dots \cap C_1(n,\mathring{U}_r). \]
Then each $x_i$ not equal to $x_0$ is the center of some closed disc $D_i$ contained in the intersection $\mathring{U}_1 \cap \dots \cap \mathring{U}_r$. Then $x$ is an element of
\[ C_1(n,\mathring{U}) \]
where $U = \Wdge_{i: x_i \neq x_0} (D_i)_+ \wdge K$ for some f-contractible $K$ contained in each $U_i$ that does not intersect any $D_i$. (Such a $K$ exists since $x_0$ has a basis of f-contractible neighbourhoods.) It follows that the open cover $\mathscr{U}$ satisfies the hypotheses of Lemma~\ref{lem:dugger/isaksen} and we deduce that the map
\[ \hocolim_{U \in \mathcal{P}} C_1(n,\mathring{U}) \to C_1(n,X) \]
is a weak equivalence. Finally, note that each inclusion $C_1(n,\mathring{U}) \to C_1(n,U)$ is a weak equivalence, from which we deduce that the map (\ref{eq:hocolim-C1}) is an equivalence.

We now turn to the bottom horizontal map in (\ref{eq:diag-hocolim}). We cannot directly apply Lemma~\ref{lem:dugger/isaksen} in the same way here because the subsets $\mathring{U}^{\smsh n} \subseteq X^{\smsh n}$ are typically not open. Instead, we apply \ref{lem:dugger/isaksen} to show that the maps
\[ \hocolim_{U \in \mathcal{P}} U^k \to X^k \]
are equivalences (where $U^k$ is the cartesian product of $k$ copies of $U$). Combining this with the natural equivalences
\[ X^{\smsh n} \homeq \thocofib_{I \subseteq \un{n}} X^I, \]
and using the commutativity of homotopy colimits, we deduce that the bottom horizontal map in (\ref{eq:diag-hocolim}) is an equivalence. This completes the proof that $\theta_X$ is an equivalence for arbitrary $X$.
\end{proof}

\begin{remark}
Taking $X = S^L$ in Proposition~\ref{prop:MX} and recalling from Example~\ref{ex:SL} that $\mathbb{M}_{S^L}$ is trivial beyond the first term, notice that we get an equivalence of $\mathbbm{Com}$-modules
\[ \Sigma^L \B(\mathbbm{1},\mathbb{E}_L,\mathbbm{Com}) \homeq (S^L)^{\smsh *}. \]
This provides a different proof of \cite[3.31]{arone/ching:2014b}, which played a key role in the calculation of the comonad that classifies Taylor towers of functors from based spaces to spectra.
\end{remark}

\begin{definition}
For $X \in \fMfld^L_*$ we set
\[ \M_X := \Sigma^\infty \mathbb{M}_X. \]
Then $\M_X$ is an $\E_L$-module and Proposition~\ref{prop:MX} implies that we have an equivalence of $\Com$-modules
\[ \B(\M_X,\E_L,\Com) \weq \Sigma^\infty X^{\smsh *}. \]
\end{definition}

The derivatives of the representable functor $R_X := \Sigma^\infty \Hom_{\based}(X,-)$ for $X \in \fMfld^L_*$ are now given by applying a form of Koszul duality to the $\E_L$-module $\M_X$. Recall from (\ref{eq:BL}) that $\B_L$ is a bisymmetric sequence that forms an $\E_L$-module in one variable and a $\K\E_L$-module in the other variable. In particular, the (derived) mapping objects $\Map_{\E_L}(\M_X,\B_L)$ inherit a $\K\E_L$-module structure from that on $\B_L$.

\begin{proposition} \label{prop:rep}
For $X \in \fMfld^L_*$, there is an equivalence of $\C$-coalgebras, natural in $X$:
\[ \der_*(R_X) \homeq \Map_{\E_L}(\M_X,\B_L). \]
In other words, the Taylor tower of $R_X$ is determined by a $\K\E_L$-module structure on its derivatives.
\end{proposition}
\begin{proof}
First recall from work of the first author \cite{arone:1999} that the Taylor tower of $R_X$ is given by
\[ (P_nR_X)(Y) \homeq \Map_{\Com_{\leq n}}(\Sigma^\infty X^{\smsh *},\Sigma^\infty Y^{\smsh *}). \]
By Proposition~\ref{prop:MX}, we can rewrite this as
\[ (P_nR_X)(Y) \homeq \Map_{(\E_L)_{\leq n}}(\M_X,\Sigma^\infty Y^{\smsh *}). \]
We now prove that the canonical map
\[ \tag{*} \Map_{(\E_L)_{\leq n}}(\M_X,\un{\E}_L) \smsh_{\E_L} \widetilde{\Sigma^\infty Y^{\smsh *}} \to \Map_{(\E_L)_{\leq n}}(\M_X,\widetilde{\Sigma^\infty Y^{\smsh *}}) \]
is an equivalence. It then follows from Theorem~\ref{thm:KE-mod-functor} that the derivatives of $P_nR_X$ have a $\K\E_L$-module structure, and that this $\K\E_L$-module is given by
\[ \Map_{(\E_L)_{\leq n}}(\M_X,\un{\E}_L) \smsh_{\E_L} \tilde{\B}_L. \]
We then claim also that there is an equivalence
\[ \tag{**} \Map_{(\E_L)_{\leq n}}(\M_X,\un{\E}_L) \smsh_{\E_L} \tilde{\B}_L \weq \Map_{(\E_L)_{\leq n}}(\M_X,\B_L). \]
The right-hand side is equivalent, as a $\K\E_L$-module, to the $n$-truncation of $\Map_{\E_L}(\M_X,\B_L)$ and we deduce therefore that the derivatives of $R_X$ are as claimed.

It remains to prove (*) and (**), and these follow from a more general result: for any cofibrant $\E_L$-module $\M'$, the map
\begin{equation} \label{eq:MX-eq} \Map_{(\E_L)_{\leq n}}(\M_X,\un{\E}_L) \smsh_{\E_L} \M' \to \Map_{(\E_L)_{\leq n}}(\M_X,\M') \end{equation}
is an equivalence. We prove this by induction on $n$ using the diagram
\[ \begin{diagram}
  \node{\Map_{\E_L}(\M_X,(\un{\E}_L)_{=n}) \smsh_{\E_L} \M'} \arrow{e} \arrow{s}
    \node{\Map_{\E_L}(\M_X,(\M')_{=n})} \arrow{s} \\
  \node{\Map_{\E_L}(\M_X,(\un{\E}_L)_{\leq n}) \smsh_{\E_L} \M'} \arrow{e} \arrow{s}
    \node{\Map_{\E_L}(\M_X,(\M')_{\leq n})} \arrow{s} \\
  \node{\Map_{\E_L}(\M_X,(\un{\E}_L)_{\leq (n-1)}) \smsh_{\E_L} \M'} \arrow{e}
    \node{\Map_{\E_L}(\M_X,(\M')_{\leq (n-1)})}
\end{diagram} \]
in which the vertical maps form fibre sequences. It is sufficient then to show that the top horizontal map here is an equivalence for all $n$. We can rewrite this map as
\begin{equation} \label{eq:hfp} \Map(\B(\M_X,\E_L,\one)(n),\un{\E}_L(n,-))^{\Sigma_n} \smsh_{\E_L} \M' \to \Map(\B(\M_X,\E_L,\one)(n),\M'(n))^{\Sigma_n}. \end{equation}
The key observation now is that the $\Sigma_n$-spectrum $\B(\M_X,\E_L,\one)(n)$ is built from finitely-many free $\Sigma_n$-cells. To see this we note that by Proposition~\ref{prop:MX} we have
\[ \B(\M_X,\E_L,\one)(n) \homeq \B(\Sigma^\infty X^{\smsh *},\Com,\one)(n) \homeq \Sigma^\infty X^{\smsh n}/\Delta^nX \]
where $\Delta^nX$ denotes the `fat diagonal' inside $X^{\smsh n}$. Since $X$ is homotopy equivalent to a based finite complex, this $\Sigma_n$-spectrum is equivalent to the suspension spectrum of a finite complex of free $\Sigma_n$-cells. It follows that in (\ref{eq:hfp}) we can replace the homotopy fixed points with homotopy orbits. These now commute with the coend and we see that the resulting map is an equivalence. This completes the proof.
\end{proof}

\begin{remark}
We stress that the $\K\E_L$-module structure on $\der_*(R_X)$ for a based space $(X,x_0)$ depends on a choice of smooth structure and framing on $X - \{x_0\}$. Different choices correspond to potentially non-equivalent $\K\E_L$-modules that yield equivalent $\C$-coalgebra structures on those derivatives.
\end{remark}

\begin{definition}
Let $G: \fMfld^L_* \to \spectra$ be a pointed simplicially-enriched functor. We define
\[ \d(G) := G(Y) \smsh_{Y \in \fMfld^L_*} \Map_{\E_L}(\M_Y,\B_{L}) \]
with $\K\E_L$-module structure on $\d(G)$ induced by that on $\B_{L}$.
\end{definition}

\begin{theorem} \label{thm:fmfld}
Let $F: \finbased \to \spectra$ be a functor that is the enriched (homotopy) left Kan extension of a pointed simplicially-enriched functor $G: \fMfld^L_* \to \spectra$ along the forgetful functor $\fMfld^L_* \to \finbased$. Then the derivatives of $F$ have a $\K\E_L$-module structure that classifies the Taylor tower. Conversely, if $F: \finbased \to \spectra$ is polynomial and the derivatives of $F$ have a $\K\E_L$-module structure, then $F$ is the Kan extension of such a $G$.
\end{theorem}
\begin{proof}
Suppose that $F$ is the homotopy left Kan extension of $G: \fMfld^L_* \to \spectra$, i.e. we have
\[ F(Y) \homeq G(Y) \smsh_{Y \in \fMfld^L_*} \Sigma^\infty \Hom_{\based}(Y,-). \]
Taking derivatives, which commutes with homotopy colimits for spectrum-valued functors, we get an equivalence of $\C$-coalgebras
\[ \der_*(F) \homeq G(Y) \smsh_{Y \in \fMfld^L_*} \der_*(R_Y) \homeq G(Y) \smsh_{Y \in \fMfld^L_*} \Map_{\E_L}(\mathbb{M}_Y,\B_{L}) = \d(G) \]
which has a $\K\E_L$-module structure as required.

The converse follows from Theorem~\ref{thm:KE-mod-functor} since $\Gamma_L$ can be identified with a full subcategory of $\fMfld^L_*$.
\end{proof}

\begin{example}
Using the description of $\mathbb{M}_{S^L}$ in Example~\ref{ex:SL}, we see that the Taylor tower of the representable functor $R_{S^L} = \Sigma^\infty \Omega^L(-)$ is determined by the $\K\E_L$-module
\[ \Sigma^{-L}\Map_{\E_L}(\one,\B_{L}). \]
\end{example}

\begin{example} \label{ex:L-mfld}
Let $M$ be any compact parallelizable $L$-dimensional manifold possibly with boundary. Then the derivatives of the functors $\Sigma^\infty \Hom_{\based}(M_+,-)$ and $\Sigma^\infty \Hom_{\based}(M/\partial M,-)$ have a $\K\E_L$-module structure. If we give $M$ a basepoint, the same is true of $\Sigma^\infty \Hom_{\based}(M,-)$.
\end{example}

We finish this section with a conjecture about one way in which the existence of a $\K\E_L$-module and $\K\E_{L-1}$-module structures on the derivatives of functors might be related.

\begin{conjecture} \label{conj:splitting}
Let $\Sigma: \based \to \based$ denote the usual reduced suspension functor. Suppose the derivatives of $F: \finbased \to \spectra$ have a $\K\E_L$-module structure. Then the derivatives of $F\Sigma$ have a $\K\E_{L-1}$-module structure.
\end{conjecture}

Evidence for this conjecture is provided by the stable splitting of mapping spaces of B\"{o}digheimer \cite{bodigheimer:1987} and others. These results imply, for example, that if $M$ is a compact framed $L$-manifold, then the Taylor tower of the functor
\[ \Sigma^\infty \Hom_{\based}(M_+,\Sigma^L-) \]
splits, i.e. has a $\K\E_0$-module structure on its derivatives. This result would follow by applying Conjecture~\ref{conj:splitting} $L$ times to Example~\ref{ex:L-mfld}. We therefore view Conjecture~\ref{conj:splitting} as a refinement of these splitting results to encompass the intermediate functors $\Sigma^\infty \Hom_{\based}(M_+,\Sigma^m-)$ for $m < L$.

\section{The Taylor tower of Waldhausen's algebraic K-theory of spaces functor} \label{sec:A}

Let $A : \based \to \spectra$ denote Waldhausen's algebraic $K$-theory of spaces functor and let $\tilde{A}$ be the corresponding reduced functor so that
\[ A(X) \homeq A(*) \times \tilde{A}(X). \]
The derivatives of $\tilde{A}$ (at $*$) were calculated by Goodwillie and can be written as
\[ \der_n(\tilde{A}) \homeq \Sigma^{1-n}\Sigma^\infty(\Sigma_n/C_n)_+ \]
where $C_n$ is the cyclic group of order $n$ sitting inside the symmetric group $\Sigma_n$.

Our goal is to show that the Taylor tower of $\tilde{A}$ (and hence that of $A$) is determined by a $\K\E_3$-module structure on these derivatives. Our strategy is to use the arithmetic square to break $\tilde{A}$ into its $p$-complete and rational parts, which we deal with separately.

Consider first the $p$-completion of $\tilde{A}$ which we write $\tilde{A}_p$. We investigate this using the relationship between $A(X)$ and the \emph{topological cyclic homology} of $X$, denoted $TC(X)$. One of the main results of \cite{bokstedt/carlsson/cohen/goodwillie/hsiang/madsen:1996} is that the difference between $A(X)$ and $TC(X)$ is locally constant. In particular, this means that the corresponding reduced theories agree in a neighbourhood of the one-point space and so the Taylor towers agree. There is therefore an equivalence of $\C$-coalgebras $\der_*(\tilde{A}_p) \homeq \der_*(\tilde{TC}_p)$.

The $p$-completion of $TC(X)$ is given by the following homotopy pullback square (see, for example, \cite{madsen:1994}):
\begin{equation} \label{eq:TC-diag} \begin{diagram}
  \node{\tilde{TC}(X)_{p}} \arrow{e} \arrow{s} \node{\Sigma \left(\Sigma^\infty \Hom_{\based}(S^1_+,X)_{p}\right)_{hS^1}} \arrow{s,r}{Tr} \\
  \node{\Sigma^\infty \Hom_{\based}(S^1_+,X)_{p}} \arrow{e,t}{1-\Delta_p} \node{\Sigma^\infty \Hom_{\based}(S^1_+,X)_{p}}
\end{diagram} \end{equation}
where the right-hand vertical map is the transfer associated to the $S^1$-action on the free loop space, and the bottom map is given by the difference between the identity map and the map induced by the \ord{$p$} power map $\Delta_p : S^1 \to S^1$.

Taking derivatives, we get a corresponding homotopy pullback of $\C$-coalgebras
\begin{equation} \label{eq:TC-diag-der} \begin{diagram}
    \node{\der_*(\tilde{TC}_p)} \arrow{e} \arrow{s} \node{\der_*(\Sigma[L_p]_{hS^1})} \arrow{s} \\
    \node{\der_*(L_p)} \arrow{e,t}{1-\Delta_p} \node{\der_*(L_p)}
\end{diagram} \end{equation}
where we are writing $L(X) := \Sigma^\infty \Hom_{\based}(S^1_+,X)$ for the suspension spectrum of the free loop space on $X$. Our goal now is to construct a model for the bottom and right-hand maps in this diagram in the category of $\K\E_3$-modules.

When $X$ is a finite CW-complex, the homotopy groups of $L(X)$ are finitely generated and so we have
\[ L_p(X) \homeq S_p \smsh L(X) \]
where $S_p$ denotes the $p$-complete sphere spectrum. It follows that $L_p$ is the left Kan extension to $\finbased$ of the functor
\[ G: \fMfld^1_* \to \spectra; G(X) = S_p \smsh \Emb^f_*(S^1_+,X) \]
and so, by Theorem~\ref{thm:fmfld}, the derivatives of $L_p$ are given by the $\K\E_1$-module
\[ S_p \smsh \der_*(L) = \der_*(L)_p. \]
By a similar argument the derivatives of $\Sigma[L_p]_{hS^1}$ are given by the $\K\E_1$-module $\Sigma[\der_*(L)_p]_{hS^1}$.

The action of $S^1$ on $L$ is determined by the action of $S^1$ on itself via translations which are framed embeddings. The induced action on $S^1$ therefore commutes with the $\K\E_1$-module structure and so induces a transfer map
\[ \Sigma[\der_*(L)_p]_{hS^1} \to \der_*(L)_p \]
which models the right-hand vertical map in (\ref{eq:TC-diag-der}).

On the other hand, the bottom map in (\ref{eq:TC-diag-der}) is a problem. The \ord{$p$} power map on $S^1$ is not a framed embedding (or even an embedding) so the induced map
\[ \Delta_p^*: \der_*(L) \to \der_*(L) \]
does not commute with the $\K\E_1$-module structure. We can solve this problem, however, by modelling the \ord{$p$} power map as a framed embedding between solid tori, i.e. as a map in $\fMfld^3_*$. This allows us to construct a map of $\K\E_3$-modules that models the bottom horizontal map in (\ref{eq:TC-diag-der}).

\begin{definition}
Let $T$ be the solid torus obtained from the cylinder $D^2 \times [0,1]$ via the identifications $(x,y,0) \sim (x,y,1)$ for $(x,y) \in D^2$. We give $T'$ the standard framing arising from that on the cylinder as a subset of $\mathbb{R}^3$.

Now recall that we have fixed a prime $p$. Let $T'$ be the subset of $T$ given by
\[ T' := \left\{ (x,y,z) \; | \; \left( x - \frac{\cos(2\pi z)}{2p} \right)^2 + \left( y - \frac{\sin(2\pi z)}{2p} \right)^2 \leq \frac{1}{p^4} \right\} \]
with the induced framing. A diagram of $T'$ in the case $p = 3$ is shown in Figure~\ref{fig:tori1}. Notice that $T'$ is a thickened helix with one twist for each revolution around the torus $T$.

\addtocounter{equation}{1}
\begin{figure}[ht]
\includegraphics[width=2in]{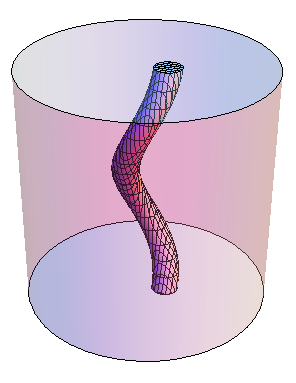}
\caption{$T'$ as a subset of $T$ for $p = 3$}
\label{fig:tori1}
\end{figure}

The inclusion $i : T' \to T$ is a framed embedding and a homotopy equivalence that commutes with the obvious collapse maps $T \to S^1$ and $T' \to S^1$ (that project onto the $z$-coordinate). In particular, the map
\[ i^* : \der_*(\Sigma^\infty \Hom_{\based}(T_+,-)) \to \der_*(\Sigma^\infty \Hom_{\based}(T'_+,-)) \]
can be used to model the identity map on $\der_*(L)$.

We define $\Delta_p : T' \to T$ by
\[ \Delta_p(x,y,z) := (px,py,pz) \]
where $pz$ is interpreted modulo $1$. Notice that $\Delta_p$ models the \ord{$p$} power map on $S^1$. The choice of thickness and radius for the helix $T'$ also ensure that $\Delta_p$ is a framed embedding. A diagram of this embedding for $p = 3$ is given in Figure~\ref{fig:tori2}. It follows that $\Delta_p$ induces a map of $\K\E_3$-modules
\[ \Delta_p^* : \der_*(\Sigma^\infty \Hom_{\based}(T_+,-)) \to \der_*(\Sigma^\infty \Hom_{\based}(T'_+,-)). \]

\addtocounter{equation}{1}
\begin{figure}[ht]
\includegraphics[width=2in]{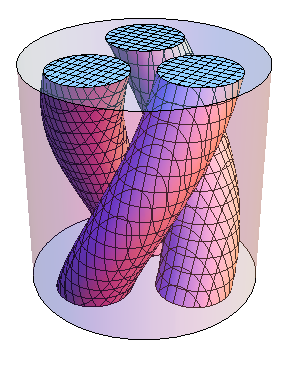}
\caption{$T'$ embedded in $T$ via $\Delta_p$ for $p = 3$}
\label{fig:tori2}
\end{figure}
\end{definition}

\begin{definition} \label{def:TC-p}
We now construct a homotopy pullback square of $\K\E_3$-modules of the form
\begin{equation} \label{eq:TC-p} \begin{diagram}
     \node{\M_p} \arrow{e} \arrow[2]{s} \node{\Sigma[\der_*(\Sigma^\infty \Hom_{\based}(T_+,-))_p]_{hS^1}} \arrow{s,r}{Tr} \\
     \node[2]{\der_*(\Sigma^\infty \Hom_{\based}(T_+,-))_p} \arrow{s,lr}{\sim}{i^*} \\
     \node{\der_*(\Sigma^\infty \Hom_{\based}(T_+,-))_p} \arrow{e,t}{i^* - \Delta_p^*} \node{\der_*(\Sigma^\infty \Hom_{\based}(T'_+,-))_p}
\end{diagram} \end{equation}
where the top right-hand map is the transfer associated to the $S^1$-action on $T$ given by action on the $z$-coordinate, and the bottom horizontal map models the difference between $i^*$ and $\Delta_p^*$ in the stable homotopy category of $\K\E_3$-modules. This defines the $\K\E_3$-module $\M_p$.
\end{definition}

The following result then follows from a comparison between diagrams (\ref{eq:TC-p}) and (\ref{eq:TC-diag-der}). We use here the fact that the forgetful functor from $\K\E_3$-modules to $\C$-coalgebras preserves homotopy pullbacks.

\begin{proposition} \label{prop:A-p}
There is an equivalence of $\C$-coalgebras
\[ \der_*(\tilde{A}_p) \homeq \M_p. \]
\end{proposition}

We now turn to the rationalization of $\tilde{A}$ which we denote by $\tilde{A}_{\mathbb{Q}}$. There is a rational equivalence
\[ \tilde{A}(X) \homeq \Sigma L(X)_{hS^1} \]
and so, by the argument used above for $p$-completion, the derivatives of the rationalization $\der_*(\tilde{A}_{\mathbb{Q}})$ are given by the $\K\E_1$-module
\[ \M'_{\mathbb{Q}} := \Sigma[\der_*(L)_{hS^1}]_{\mathbb{Q}} \]
which, via the operad map $\K\E_3 \to \K\E_1$, we also view as a $\K\E_3$-module.

In order to form the arithmetic square for $\tilde{A}$, we need also a model for the derivatives of the functor
\[ \left[ \prod_p \tilde{A}_p \right]_{\mathbb{Q}}. \]
On the one hand, since rationalization commutes with taking derivatives, Proposition~\ref{prop:A-p} implies that there is an equivalence of $\C$-coalgebras
\[ \der_*\left( \left[ \prod_p \tilde{A}_p \right]_{\mathbb{Q}} \right) \homeq \hat{\M}_{\mathbb{Q}} := \left[ \prod_p \M_p \right]_{\mathbb{Q}}. \]
On the other hand, an argument similar to that of the previous paragraph gives us an equivalence of $\C$-coalgebras
\[ \der_*\left( \left[ \prod_p \tilde{A}_p \right]_{\mathbb{Q}} \right) \homeq \hat{\M}'_{\mathbb{Q}} := \left[ \prod_p \Sigma[\der_*(L)_p]_{hS^1}\right]_{\mathbb{Q}}. \]
We now have a diagram of the form
\begin{equation} \label{eq:arithmetic} \begin{diagram}
  \node[3]{\prod_p \M_p} \arrow{s} \\
  \node[3]{\hat{\M}_{\mathbb{Q}}} \arrow{s,r}{\sim} \\
  \node{\M'_{\mathbb{Q}}} \arrow{e} \node{\hat{\M}'_{\mathbb{Q}}} \arrow{ne,..} \arrow{e,b}{\sim} \node{\der_*\left( \left[ \prod_p \tilde{A}_p \right]\right)}
\end{diagram} \end{equation}
where the top vertical map (given by rationalization) and the left-hand horizontal map (induced by $p$-completion for each $p$) are maps of $\K\E_3$-modules.

In order to form the pullback in the category of $\K\E_3$-modules, we need to show that there is an equivalence of $\K\E_3$-modules $\hat{\M}'_{\mathbb{Q}} \weq \hat{\M}_{\mathbb{Q}}$ (the dotted arrow above) that lifts the given equivalence of $\C$-coalgebras. We do this by showing that \emph{any} morphism of $\C$-coalgebras of this form lifts to a morphism of $\K\E_3$-modules (which is then necessarily an equivalence since these are detected on the underlying symmetric sequences).

Let $\widetilde{\Map}_{\K\E_3}(-,-)$ and $\widetilde{\Map}_{\C}(-,-)$ denote the derived mapping spectra for $\K\E_3$-modules and $\C$-coalgebras respectively.

\begin{lemma} \label{lem:surj}
The map
\[ \pi_0 \; \widetilde{\Map}_{\K\E_3}(\hat{\M}'_{\mathbb{Q}},\hat{\M}_{\mathbb{Q}}) \to \pi_0 \; \widetilde{\Map}_{\C}(\hat{\M}'_{\mathbb{Q}},\hat{\M}_{\mathbb{Q}}), \]
induced by the map $\C_{\K\E_3} \to \C$ of comonads, is surjective.
\end{lemma}
\begin{proof}
In \cite{arone/ching:2014} we showed that the comonad $\C$ encodes the structure for a `divided power' module over the operad $\K\Com$ that is Koszul dual to the commutative operad. In particular, there is a map of comonads $\C \to \C_{\K\Com}$ that is built from norm maps of the form
\[ \Map(\K\Com(n_1) \smsh \dots, \A(n))_{h\Sigma_{n_1} \times \dots \times \Sigma_{n_r}} \arrow{e,t}{N} \Map(\K\Com(n_1) \smsh \dots, \A(n))^{h\Sigma_{n_1} \times \dots \times \Sigma_{n_r}}. \]
These norm maps are equivalences when $\A(n)$ is a rational spectrum, and, in that case, the terms of the symmetric sequence $\C(\A)$ are again rational. Since the module $\hat{\M}_{\mathbb{Q}}$ is formed from rational spectra, it follows that there is an equivalence
\[ \widetilde{\Map}_{\C}(\hat{\M}'_{\mathbb{Q}},\hat{\M}_{\mathbb{Q}}) \weq \widetilde{\Map}_{\K\Com}(\hat{\M}'_{\mathbb{Q}},\hat{\M}_{\mathbb{Q}}) \]
induced by the map $\C \to \C_{\K\Com}$. It is therefore sufficient to show that the map
\[ \tag{*} \pi_0 \; \widetilde{\Map}_{\K\E_3}(\hat{\M}'_{\mathbb{Q}},\hat{\M}_{\mathbb{Q}}) \to \pi_0 \; \widetilde{\Map}_{\K\Com}(\hat{\M}'_{\mathbb{Q}},\hat{\M}_{\mathbb{Q}}) \]
is a surjection, where we can think of this as induced by the map of operads $\K\Com \to \K\E_3$ that is dual to the standard map $\E_3 \to \Com$. In other words, it is sufficient to show that any (derived) morphism of $\K\Com$-modules lifts, up to homotopy, to a morphism of $\K\E_3$-modules.

To prove this, we construct a square of the following form:
\begin{equation} \label{eq:bkss} \begin{diagram}
  \node{\pi_0 \; \widetilde{\Map}_{\K\E_3}(\hat{\M}'_{\mathbb{Q}},\hat{\M}_{\mathbb{Q}})} \arrow{e,t}{\text{(*)}} \arrow{s,l,A}{\text{(A)}}
    \node{\pi_0 \; \widetilde{\Map}_{\K\Com}(\hat{\M}'_{\mathbb{Q}},\hat{\M}_{\mathbb{Q}})} \arrow{s,lr}{\isom}{\text{(B)}} \\
  \node{\Hom_{H_*\K\E_3}(H_*\hat{\M}'_{\mathbb{Q}},H_*\hat{\M}_{\mathbb{Q}})} \arrow{e,tb}{\text{(C)}}{\isom}
    \node{\Hom_{H_*\K\Com}(H_*\hat{\M}'_{\mathbb{Q}},H_*\hat{\M}_{\mathbb{Q}}).}
\end{diagram} \end{equation}
The terms in the bottom row are the homomorphism of modules over the operads of graded $\mathbb{Q}$-vector spaces given by the rational homology of the operads $\K\E_3$ and $\K\Com$, and of the modules $\hat{\M}'_{\mathbb{Q}}$ and $\hat{\M}_{\mathbb{Q}}$. The map labelled (C) is induced by the operad map. The maps labelled (A) and (B) are edge homomorphisms associated with the Bousfield-Kan spectral sequences used to calculated the homology of the mapping spectra in the top row, as we explain below. Given that the diagram is commutative, it will then be sufficient to show that (A) is a surjection, and that (B) and (C) are isomorphisms.

We start by analyzing the Bousfield-Kan spectral sequence for calculating the homotopy groups of $\widetilde{\Map}_{\K\E_3}(\hat{\M}'_{\mathbb{Q}},\hat{\M}_{\mathbb{Q}})$. The $E^1$ term of this spectral sequence takes the form
\[ E^1_{-s,t} \isom \pi_{t-s} \Map_{\Sigma}(\mathsf{nondeg}[\hat{\M}'_{\mathbb{Q}} \circ \underbrace{\K\E_3 \circ \dots \circ \K\E_3}_{\text{$s$ factors}}], \hat{\M}_{\mathbb{Q}}) \]
where $\mathsf{nondeg}$ means that we take the `nondegenerate' terms in the given iterated composition product, i.e. where each of the factors $\K\E_3$ contributes something from a term higher than the first.

\begin{claim}
The $E^2$ term of the spectral sequence is given by
\[ E^2_{-s,t} \isom \Ext^{s,t}_{H_*\K\E_3}(H_*\hat{\M}'_{\mathbb{Q}},H_*\hat{\M}_{\mathbb{Q}}). \]
\end{claim}
\begin{proof}
In the formula for the $E^1$-term $\Map_{\Sigma}(-,-)$ denotes the mapping spectrum for symmetric sequences, which we can expand out (assuming cofibrant models for $\K\E_3$ and $\hat{\M}'_{\mathbb{Q}}$) as
\[ E^1_{-s,t} \isom \prod_n \pi_{t-s} \Map(\mathsf{nondeg}[\hat{\M}'_{\mathbb{Q}} \circ \underbrace{\K\E_3 \circ \dots \circ \K\E_3}_{\text{$s$ factors}}](n), \hat{\M}_{\mathbb{Q}}(n))^{h\Sigma_n}. \]
For a rational spectrum $X$ with action of a finite group $G$, we have
\[ \pi_*(X^{hG}) \isom \pi_*(X_{hG}) \isom [\pi_*X]_G \isom [\pi_*X]^G \]
and so we can write
\[ E^1_{-s,t} \isom \prod_n \left[ \pi_{t-s} \Map(\mathsf{nondeg}[\hat{\M}'_{\mathbb{Q}} \circ \underbrace{\K\E_3 \circ \dots \circ \K\E_3}_{\text{$s$ factors}}](n), \hat{\M}_{\mathbb{Q}}(n))\right]^{\Sigma_n}. \]
The homotopy groups of a rational spectrum are equivalent to its rational homology groups and so, using the K\"{u}nneth Theorem, we get
\[ E^1_{-s,t} \isom \prod_n \left[ \Hom(\mathsf{nondeg}[H_*\hat{\M}'_{\mathbb{Q}} \circ \underbrace{H_*\K\E_3 \circ \dots \circ H_*\K\E_3}_{\text{$s$ factors}}](n), H_*\hat{\M}_{\mathbb{Q}}(n))\right]_{t-s}^{\Sigma_n}. \]
We can now see that $(E^1,d^1)$ is precisely the complex used to calculate $\Ext$ in the category of modules over the operad $H_*\K\E_3$ of graded $\mathbb{Q}$-vector spaces.
\end{proof}

\begin{claim}
The Bousfield-Kan spectral sequence for the homotopy groups of $\widetilde{\Map}_{\K\E_3}(\hat{\M}'_{\mathbb{Q}},\hat{\M}_{\mathbb{Q}})$ collapses at $E^2$ and so the edge homomorphism
\[ \pi_0 \; \widetilde{\Map}_{\K\E_3}(\hat{\M}'_{\mathbb{Q}},\hat{\M}_{\mathbb{Q}}) \to E^2_{0,0} = \Hom_{H_*\K\E_3}(H_*\hat{\M}'_{\mathbb{Q}},H_*\hat{\M}_{\mathbb{Q}}) \]
is a surjection.
\end{claim}
\begin{proof}
First note that, since $\hat{\M}'_{\mathbb{Q}}$ and $\hat{\M}_{\mathbb{Q}}$ are rational, the spectral sequence for calculating the homotopy groups of $\widetilde{\Map}_{\K\E_3}(\hat{\M}'_{\mathbb{Q}},\hat{\M}_{\mathbb{Q}})$ is isomorphic to that for calculating the homotopy groups of $\widetilde{\Map}_{\K\E_3 \smsh H\mathbb{Q}}(\hat{\M}'_{\mathbb{Q}},\hat{\M}_{\mathbb{Q}})$ where we have replaced $\K\E_3$ with its rationalization.

We now use rational formality of the operad $\K\E_3$. Recall that Kontsevich proved in \cite{kontsevich:1999} that the operads $\mathbb{E}_L$ are formal over the real numbers, and that Guill\'{e}n Santos et al. proved in \cite{guillensantos/navarro/pascual/roig:2005} that formality over $\mathbb{R}$ descends to formality over $\mathbb{Q}$ for operads with no zero term. This means that there is an equivalence of operads
\[ \E_L \smsh H\mathbb{Q} \homeq H\E_L \]
where the right-hand side denotes the operad formed by the generalized Eilenberg-MacLane spectra associated to the rational homology groups of $\E_L$. Taking Koszul duals commutes with rationalization and so we have an equivalence
\[ \K\E_L \smsh H\mathbb{Q} \homeq \K(H\E_L). \]
Koszul duality for an operad of Eilenberg-MacLane spectra reduces to Koszul duality of the underlying operad of graded abelian groups. Getzler and Jones \cite{getzler/jones:1994} proved that the homology of $\E_L$ is a `Koszul' operad in the sense of Ginzburg and Kapranov \cite{ginzburg/kapranov:1994} from which it follows that
\[ \K(H\E_L) \homeq H\K\E_L. \]
Altogether we have an equivalence
\[ \K\E_L \smsh H\mathbb{Q} \homeq H\K\E_L, \]
so, in particular, $\K\E_3$ is rationally formal.

The modules $\hat{\M}'_{\mathbb{Q}}$ and $\hat{\M}_{\mathbb{Q}}$ are already Eilenberg-MacLane spectra (since the derivatives of $\tilde{A}$ are wedges of copies of sphere spectra) and so the spectral sequence in question is equivalent to that used to calculate the homotopy groups of the mapping spectrum
\[ \widetilde{\Map}_{H\K\E_3}(H\hat{\M}'_{\mathbb{Q}},H\hat{\M}_{\mathbb{Q}}), \]
or, equivalently, that of the corresponding object in the world of rational chain complexes
\[ \widetilde{\Map}_{H_*\K\E_3}(H_*\hat{\M}'_{\mathbb{Q}},H_*\hat{\M}_{\mathbb{Q}}). \]
But the differentials in the underlying chain complexes here are all trivial so this spectral sequence collapses at $E_2$ as claimed.
\end{proof}

This shows that the map (A) in diagram (\ref{eq:bkss}) is a surjection. We now turn to the corresponding spectral sequence for the homotopy groups of $\widetilde{\Map}_{\K\Com}(\hat{\M}'_{\mathbb{Q}},\hat{\M}_{\mathbb{Q}})$.

\begin{claim}
The $E^2$-term of this spectral sequence is given by
\[ E^2_{-s,t} \isom \Ext^{s,t}_{H_*\K\Com}(H_*\hat{\M}'_{\mathbb{Q}},H_*\hat{\M}_{\mathbb{Q}}) \]
and this is zero for $t \neq 0$. The spectral sequence therefore collapses and the edge homomorphism is an isomorphism
\[ \pi_0 \; \widetilde{\Map}_{\K\Com}(\hat{\M}'_{\mathbb{Q}},\hat{\M}_{\mathbb{Q}}) \isom E^2_{0,0} \isom Hom_{H_*\K\Com}(H_*\hat{\M}'_{\mathbb{Q}},H_*\hat{\M}_{\mathbb{Q}}). \]
\end{claim}
\begin{proof}
The same argument as for $\K\E_3$ gives the form of the $E^2$-term. But the rational homologies $H_*\K\Com(n)$, $H_*\hat{\M}'_{\mathbb{Q}}(n)$ and $H_*\hat{\M}_{\mathbb{Q}}(n)$ are all concentrated in a single degree ($1-n$). It is then easy to see directly that $E^1_{-s,t} = 0$ for $t \neq 0$.
\end{proof}

This shows that the map (B) in diagram (\ref{eq:bkss}) is an isomorphism, and the diagram commutes by naturality of the construction of the spectral sequence.

It remains to check that the map (C) is an isomorphism. For this, we notice that, for degree reasons, the action of the operad $H_*\K\E_3$ on the modules $H_*\hat{\M}'_{\mathbb{Q}}$ and $H_*\hat{\M}_{\mathbb{Q}}$ is determined fully by what happens on the terms $H_{1-n}(\K\E_3(n))$. Since the maps
\[ H_{1-n}(\K\Com(n)) \to H_{1-n}(\K\E_3(n)) \]
are isomorphisms (each side is isomorphic to $\mathsf{Lie}(n)$), we can see directly that a morphism of $H_*\K\E_3$-modules of the form
\[ H_*\hat{\M}'_{\mathbb{Q}} \to H_*\hat{\M}_{\mathbb{Q}} \]
carries precisely the same data as a morphism of $H_*\K\Com$-modules. Thus the map (C) is an isomorphism as required. This completes the proof of Lemma~\ref{lem:surj}.
\end{proof}

We then build a homotopy pullback square of $\K\E_3$-modules of the form
\begin{equation} \label{eq:A} \begin{diagram}
   \node{\M} \arrow[2]{e} \arrow{s} \node[2]{\prod_{p} \M_p} \arrow{s} \\
   \node{\M'_{\mathbb{Q}}} \arrow{e} \node{\hat{\M}'_{\mathbb{Q}}} \arrow{e,b}{\sim} \node{\hat{\M}_{\mathbb{Q}}}
\end{diagram} \end{equation}
and obtain the following result.

\begin{theorem} \label{thm:A}
There is an equivalence of $\C$-coalgebras
\[ \der_*(\tilde{A}) \homeq \M. \]
In other words, the derivatives of the functor $\tilde{A}$ are a $\K\E_3$-module.
\end{theorem}
\begin{proof}
The arithmetic square for $\tilde{A}$ gives us a homotopy pullback of $\C$-coalgebras
\[ \begin{diagram}
  \node{\der_*\tilde{A}} \arrow{e} \arrow{s} \node{\der_*\left(\prod_p \tilde{A}_p\right)} \arrow{s} \\
  \node{\der_*\tilde{A}_{\mathbb{Q}}} \arrow{e} \node{\der_*\left(\prod_p \tilde{A}_p\right)_{\mathbb{Q}}}
\end{diagram} \]
Comparison with (\ref{eq:A}) gives us the required equivalence.
\end{proof}

\begin{corollary}
The Taylor tower of $A$ has the following descriptions:
\begin{enumerate}
  \item there is an $\E_3$-comodule $\N_n$ (Koszul dual to the $\K\E_3$-module $\der_{\leq n}(\tilde{A}))$ such that
  \[ P_nA(X) \homeq A(*) \times [\N_n \smsh_{\E_3} \Sigma^\infty X^{\smsh *}] \]
  with the Taylor tower map $P_nA \to P_{n-1}A$ induced by a map of $\E_3$-comodules $\N_n \to \N_{n-1}$;
  \item there is a functor $G_n: \fMfld^3_* \to \spectra$ such that $P_nA$ is the enriched homotopy left Kan extension of $G_n$ along the forgetful functor $\fMfld^3_* \to \finbased$ with the Taylor tower map $P_nA \to P_{n-1}A$ induced by a natural transformation $G_n \to G_{n-1}$.
\end{enumerate}
\end{corollary}

\begin{remark}
We do not know if our result for the Taylor tower of $A$ is the best possible. It is conceivable that the derivatives of $\tilde{A}$ have, in fact, a $\K\E_2$ or $\K\E_1$-module structure. The Taylor tower of $A(\Sigma X)$ is known to split by work of B{\"o}kstedt et al. \cite{bokstedt/carlsson/cohen/goodwillie/hsiang/madsen:1996}, so, given Conjecture~\ref{conj:splitting}, one might guess that the latter is true.
\end{remark}

\bibliographystyle{amsplain}
\bibliography{mcching}

\end{document}